\documentclass[12pt]{amsart}
\usepackage{graphicx,caption,subcaption}
\usepackage{tikz}        
\usepackage{amssymb}
\usepackage{epstopdf}
\usepackage[mathscr]{eucal}
\usepackage{amsmath,amssymb,amscd}
\usepackage{color}
\usepackage{epsfig}
\usepackage{bbold}
\usepackage{hyperref}
\newtheorem{theorem}{Theorem}
\newtheorem{lemma}{Lemma}
\newtheorem{conjecture}{Conjecture}
\newtheorem{definition}{Definition}

\newtheorem{proposition}{Proposition}
\theoremstyle{definition}

\newtheorem{example}{Example}

\def \mb{\mathbb}

\def \bf{\mathbf}

\def \R{\mb R}                 

\def \and{\mbox{and}}
\usepackage[top=4cm, bottom=4cm, left=3.5cm, right=3.5cm]{geometry}
\title{Global Existence and Singularity of the N-body Problem with Strong Force}
\date{\today}

\begin{document}

\maketitle

\markboth{Yanxia Deng, Slim Ibrahim}{Global Existence and Singularity of the N-body Problem}
\author{\begin{center} 
Yanxia Deng\footnote{Department of Mathematics and Statistics, University of Victoria, Victoria, BC, Canada \quad yd17@uvic.ca}, Slim Ibrahim\footnote{Department of Mathematics and Statistics, University of Victoria, Victoria, BC, Canada \quad ibrahims@uvic.ca}\\

\vspace{5mm}
(In memory of Florin Diacu)
\end{center}}

\begin{abstract}
We use the idea of ground states and excited states in nonlinear dispersive equations (e.g. Klein-Gordon and Schr\"odinger equations) to characterize solutions in the N-body problem with strong force under some energy constraints. Indeed, relative equilibria of the N-body problem play a similar role as solitons in PDE. We introduce the ground state and excited energy for the N-body problem. {We are able to give a conditional dichotomy of the global existence and singularity below the excited energy in Theorem \ref{thm:dichotomy}, the proof of which seems original and simple. This dichotomy is given by the sign of a threshold function $K_\omega$}. The characterization for the two-body problem in this new perspective is non-conditional and it resembles the results in PDE nicely. For $N\geq3$, we will give some refinements of the characterization, in particular, we examine the situation where there are infinitely transitions for the sign of $K_\omega$.
\end{abstract}

\section{Introduction}
\subsection{Background and motivation}
Unstable dispersive Hamiltonian evolution equations, such as semi-linear Klein-Gordon and Schr\"odinger equations, exhibit ``soliton"-like solutions {which correspond to relative equilibria in the $N$-body problem}. Amongst those one singles out the ground state, which has the lowest energy of all solitons. When {the energy of solutions is slightly above the ground state energy threshold} 
one obtains a trichotomy in forward time for this regime of energies:
\begin{itemize}
\item[(i)] finite time blow-up;
\item[(ii)] scattering to zero;
\item[(iii)] convergence to the ground states.
\end{itemize}
The same holds in backward time, and all nine combinations allowed by the forward/backward time trichotomy can occur. (cf. \cite{AkIb18} \cite{NaSc91} \cite{Na17} \cite{PaSa75})

In this paper, we study if this mechanism can be applied to the N-body problem. In particular, we consider N point particles moving in the Euclidean space $\mathbb{R}^3$. The mass and the position of the $i^{th}$ particle is $m_i>0$ and $x_i\in\R^3$, and let $\dot{x}_i$ be its velocity.  The potential is equal to
\begin{equation}\label{gen_po}U(\bf x)=-\sum_{i<j}\frac{m_im_j}{|x_i-x_j|^\alpha}\,, \quad \alpha>0.\end{equation} 
The potential $U(\bf x)$ is homogeneous with degree $-\alpha$. When $\alpha=1$, $U$ is the classical Newtonian gravitational potential. When $\alpha\geq 2$, it is usually known as the ``strong force" problem \cite{Gor75}. There are strong force examples in physics, for example, the Lennard-Jones potential which models interaction between a pair of neutral atoms or molecules contains terms with $\alpha=6$ and $\alpha=12$ \cite{LJ24}. {In fact, the Lennard-Jones potential is a quasi-homogeneous function \begin{equation}U_{\mathrm{LJ}}(r):=-\frac{A}{r^6}+\frac{B}{r^{12}}, \quad A, B>0.\end{equation} where $r$ is the distance between two mass points. In the current paper, we would like to focus on the homogeneous potential N-body problem only, and the generalizations to quasi-homogenous N-body problem will be investigated in future work following ideas from \cite{AkIbKi18} \cite{AkIbIk18} \cite{AkIb18}.}

The motion of the N-body is governed by the differential equation:
\begin{equation}
\label{eq:nbody}
m_i\ddot{x_i}=-\nabla_i U(\bf x)=-\alpha\sum_{j\neq i}\frac{m_im_j(x_i-x_j)}{|x_i-x_j|^{\alpha+2}},\quad i=1,\cdots, N.
\end{equation}
It is a Hamiltonian system, and $\bf x(t)$ enjoys the conservation of energy \begin{equation}E(\bf x, \dot{\bf x}):=\frac{1}{2}\sum_{i=1}^Nm_i|\dot{x}_i|^2+U(\bf x),\end{equation}  the angular momentum \begin{equation}A(\bf x, \dot{\bf x}):=\sum_{i=1}^N m_i x_i\times \dot{x}_i,\end{equation} {and the linear momentum \begin{equation}M(\bf x, \dot{\bf x}):=\sum_{i=1}^N m_i\dot{x}_i.\end{equation}}
Let \begin{equation}\nonumber
\begin{split}
\Delta_{ij}&=\{\bf x=(x_1, \cdots, x_N)\in (\R^3)^N | x_i=x_j\},\\
\Delta&=\bigcup_{i<j}\Delta_{ij}.
\end{split}
\end{equation}
{Then, the potential $U$ is a real-analytic function on $(\R^3)^N\setminus \Delta$, and for given $\bf x(0)\in (\R^3)^N\setminus \Delta$ and $\dot{\bf x}(0)\in (\R^3)^N$, there exists a unique solution $\bf x(t)$ defined on $[0, \sigma)$, where $\sigma$ is maximal. }

\begin{definition}[Global existence and singularity]
If $\sigma<\infty$, the solution $\bf x(t)$ is said to experience a singularity at $\sigma$. Otherwise, we say $\bf x(t)$ exists globally. 
\end{definition}

Determining what constitutes a singularity of the N-body problem has been a long-standing problem in celestial mechanics. The first major result is known as the Painlev{\'e}'s theorem; asserting that the minimum distance between all pairs of particles must approach zero at the singularity. The proof of Painlev{\'e}'s theorem works for the $\alpha-$potential without intrinsic difficulty (cf. \cite{SiMo71} \cite{SaaXia96}). More precisely, let $d(\bf x, \Delta)$ be the distance of the point $\bf x$ to the set $\Delta$, we have

\begin{theorem}[Painlev{\'e}]
\label{thm:pain}
If $\bf x(t)$ is a solution to the N-body problem (\ref{gen_po})(\ref{eq:nbody}), and experiences a singularity at $t=\sigma$, then \[d(\bf x(t), \Delta)\to 0, \quad as \,\, t\to\sigma.\]
\end{theorem}

Painlev{\'e}'s theorem makes it natural to ask whether $\bf x(t)$ must approach a definite point on $\Delta$ as $t\to \sigma$. We have the following definition.
\begin{definition}[Collision vs non-collision singularity]
If $\bf x(t)$ approaches a definite point in $\Delta$ as $t\to \sigma$, the singularity is called a collision singularity. Otherwise the singularity is called a non-collision singularity. 
\end{definition}

The existence of collision singularity is more or less trivial. For example, a homothetic solution with a total collision. Also, binary collisions in the collinear N-body problem are inevitable, since the configuration space is highly restricted.  However, the existence of non-collision singularity for the Newtonian N-body problem was remained open for about 100 years until Xia \cite{xia92} gave the first affirmative answer in the early 1990s. An important difference between collision and non-collision singularities was given by von Zeipel \cite{Zei1908}. More precisely, let the moment of inertia be $$I(\bf x):=\sum_{i=1}^Nm_i|x_i|^2,$$ which measures the size of the system, then 

\begin{theorem}[von Zeipel]
If $\sigma$ is a singularity, and $\lim_{t\to \sigma} I<\infty$, then $\sigma$ is a collision singularity. On the other hand, if $\sigma$ is a non-collision singularity, then $\lim_{t\to \sigma} I=\infty$
\end{theorem}

There are intricate relationships between collision and non-collision singularities. Saari and Xia \cite{SaaXia96} proposed a conjecture that the set of points leading to non-collision singularities precisely corresponds to the set of all accumulation points in the extended phase space. They proved some weaker results to support the validity of the conjecture. We refer the reader to \cite{SaaXia96} and its references for a more complete survey on this question. 

Our main goal in this paper is to characterize the set of initial conditions yielding global solutions or singular solutions under some energy threshold constraints. That is, we are interested in determining the finiteness/infiniteness of $\sigma$ based on the constraints of the initial conditions, and we do not care about the eventual chaotic process of the solution. Indeed, solutions which are known as relative equilibria (cf. Definition \ref{def:RE}) seem to play an important role in such characterizations. These ideas have been extensively exploited in PDEs.

\subsection{Ground state energy and excited energy}
In the current paper, the energy will be considered to be smaller than some critical value, which we shall call it the ground state energy. The first task is to define the ground state energy. The Lagrange-Jacobi identity for the N-body problem is
\begin{equation}
\begin{split}
\frac{d^2}{dt^2}I(\bf x(t))&=2\sum_{i=1}^3m_i|\dot{x_i}(t)|^2+2\sum_{i=1}^3m_ix_i(t)\cdot\ddot{x_i}(t),\\
&=2\sum_{i=1}^3m_i|\dot{x_i}(t)|^2-2 \bf x \cdot \nabla U,\\
&=2\sum_{i=1}^3m_i|\dot{x_i}(t)|^2+2\alpha U(\bf x),\\
&=4[E(\bf x, \dot{\bf x})+(\alpha/2-1)U(\bf x)].
\end{split}
\end{equation}

{Let $V(\bf x, \dot{\bf x}):=E(\bf x, \dot{\bf x})+(\alpha/2-1)U(\bf x)$, we define the ground state energy as}

\begin{definition}[Ground state energy] 
\begin{equation}
\nonumber E^\star:=\inf \{E(\bf x, \dot{\bf x})| V(\bf x, \dot{\bf x})=0\}.
\end{equation}
\end{definition}

This minimizing problem is trivial. When $V(\bf x, \dot{\bf x})=0$, we have \[E(\bf x, \dot{\bf x})=-(\alpha/2-1)U(\bf x),\] and it is equivalent to \[\frac{1}{2}\sum_{i=1}^Nm_i|\dot{x}_i|^2=-\frac{\alpha}{2}U(\bf x).\] The supreme of $-U(\bf x)$ under the constraint $V(\bf x, \dot{\bf x})=0$ is $\infty$ and the infimum of $-U(\bf x)$ under the constraint is 0. So $E^\star=-\infty$ for $\alpha<2$ and $E^\star=0$ for $\alpha\geq2$. When $\alpha>2$, the ground state energy $E^\star=0$ is attained by the special state where all bodies are at infinity with zero velocity, and we will call it the ground state.

{Since $E^\star=-\infty$ for $\alpha<2$, which is not applicable when we consider solutions below the ground state energy, we will focus on the strong force case.} In fact, Saari \cite{Saa71} \cite{Saa73} showed that it is improbable in the sense of Lebesgue measure to have collisions for the Newtonian ($\alpha=1$) gravitational system.

\begin{theorem}[Saari, 1971-1973, \cite{Saa71}\cite{Saa73}]
The set of initial conditions for Newtonian N-body problem leading to collisions has Lebesgue measure zero in the phase space. 
\end{theorem}

More recently, Fleischer and Knauf \cite{FlKn18} extended Saari's improbability theorem to $0<\alpha<2$. We remark that Saari's improbability theorem holds for collision singularities of the N-body problem for $0<\alpha<2$, and very likely for all singularities of the problem. { This is another motivation that why we do not consider $\alpha<2$. }When $\alpha\geq2$, the collision set has positive Lebesgue measure {as any solution with negative energy has a collision. Indeed, for $\alpha\geq2$ any solution with negative energy satisfies $\ddot{I}\leq4E<0$ implying that $I(\bf x(t))$ is less than a concave downward parabola and must become negative for $t\geq t^*$, where $t^*$ is some positive finite number. But $I(\bf x)$ is always nonnegative, thus $\sigma\leq t^*<\infty$.}

{When $\alpha>2$, based on the Lagrange-Jacobi identity we observe that there is room for positive energy solutions to experience singularities.} To go beyond the zero energy, we seek new ways to define the next threshold energy. The appropriate candidates are the relative equilibria in the N-body problem. 
\begin{definition}[Relative equilibrium, cf. \cite{MeOf17}]
\label{def:RE}
A solution $\bf x(t)=(x_1(t), \cdots, x_N(t))$ of the N-body problem is called a relative equilibrium if there exists $\Omega(t)\in SO(3)$ such that \[x_i(t)=\Omega(t)x_i(0),\]
for all $i=1, \cdots, N$. 
\end{definition}

A relative equilibrium of the N-body problem is a solution where the configuration remains an isometry of the initial configuration, as if the configuration was a rigid body. They are particular cases of homographic solutions. {It is well-known that relative equilibria of the N-body problem are planar solutions (cf. Proposition \ref{prop:planarRE}). We refer the readers to \cite{Wintner} for a more in depth introduction of homographic solutions and relative equilibria.}

From the Principal Axis Theorem in geometry, a one-parameter subgroup in $SO(3)$ has the from \[\Omega(t)=P\begin{pmatrix}\cos(\omega t)&-\sin(\omega t)&0\\ \sin(\omega t)&\cos(\omega t)&0\\0&0&1\end{pmatrix}P^{-1}, \quad P\in SO(3).\] 
 Without loss of generality, let's assume \[R_{\omega}(t)\bf q=(R_{\omega}(t)q_1, \cdots, R_{\omega}(t)q_N),\] is a relative equilibrium, where $R_{\omega}(t)=\begin{pmatrix}\cos \omega t &-\sin \omega t &0\\\sin\omega t& \cos\omega t&0\\0&0&1\end{pmatrix}$, and $\bf q$ is the initial configuration with $q_{i3}=0$ for $i=1, \cdots, N$. {Note that we have used $\bf q$ instead of $\bf x$ to denote the special initial configurations that lead to relative equilibria. When using $\bf x$, we mean a general configuration point.}

If $R_{\omega}(t)\bf q$ is a relative equilibrium, then \begin{equation}\label{gen_re}\nabla (\frac{\omega^2}{2}\sum_{i=1}^Nm_i|q_i|^2-U(\bf q))=\bf 0. \end{equation}

The energy of the relative equilibrium is
\begin{equation}E_\omega(\bf q):=E(R_{\omega}(t)\bf q, \dot{(R_{\omega}(t)\bf q}))=\frac{\omega^2}{2}\sum_{i=1}^Nm_i|q_i|^2+U(\bf q).\end{equation}

Now for each fixed frequency parameter $\omega>0$, we define a function \[K_\omega: (\R^3)^N\setminus \Delta\to \R,\] where \begin{equation}K_\omega(\bf x):=\bf x\cdot \nabla (\frac{\omega^2}{2}\sum_{i=1}^Nm_i|x_i|^2-U(\bf x))=\omega^2I(\bf x)+\alpha U(\bf x). \end{equation}

{Note that there are infinitely many such functions $K_\omega$ in the PDE analogue \cite{IbMaNa11}, here our $K_\omega$ is the special case \[K_\omega(\bf x)=-\frac{d}{d\lambda}(U_{\mathrm{eff}}(\lambda\bf x))|_{\lambda=1},\]
where $U_{\mathrm{eff}}(\bf x):=-(\frac{\omega^2}{2}I(\bf x)-U(\bf x))$ is known as the effective potential.}

Let \begin{equation}E_\omega(\bf x):=\frac{\omega^2}{2}\sum_{i=1}^Nm_i|x_i|^2+U(\bf x), \quad \bf x\in(\R^3)^N\setminus \Delta.\end{equation}

\begin{definition}[Excited energy] \begin{equation}E^*(\omega):=\inf\{E_\omega(\bf x): K_\omega(\bf x)=0\}.\end{equation}
We call $E^*(\omega)$ the excited energy. It only depends on the frequency $\omega$. 
\end{definition}

{The motivation of $E^*(\omega)$ is ``the lowest energy among all relative equilibria with a fixed frequency $\omega$". However, note that if $\Omega(t)\bf q$ is a relative equilibrium, then $K_\omega(\bf q)=0$, but the reverse is not true. That is, if $\bf x$ is a configuration satisfying $K_\omega(\bf x)$=0, then $\bf x$ may not lead to a relative equilibrium for any choice of the initial velocities. This is because there are non-planar configurations (e.g. equal mass 4-body tetrahedron configuration) satisfying $K_\omega(\bf x)=0$, while every relative equilibrium in $\R^3$ must be planar. More details could be found in section 3.1. As a consequence, the set of configurations $\{\bf x: K_\omega(\bf x)=0\}$ is larger than the configuration set of relative equilibria with frequency $\omega$. We will show that for $\alpha>2$ the minimum of $E_\omega(\bf x)$ under the constraint $K_\omega(\bf x)=0$ is achieved by central configurations (cf. Proposition \ref{prop:EEcc}).} Central configurations are those $\bf x$ satisfying the equation $\nabla U_{\mathrm{eff}}(\bf x)=\bf 0$. The finiteness of the number of central configurations (modulo rotation and dilation symmetry) is another big problem in the N-body problem. It was listed as a problem for the twenty-first century by Smale \cite{smale00}, and we refer the readers to the rich literature out there.

When $K_\omega(\bf x)=0$, we have 
\begin{equation}\nonumber E_\omega(\bf x)=-(\frac{\alpha}{2}-1)U(\bf x).\end{equation}

It is not hard to show that (cf. Lemma \ref{lem:ground}),
\begin{itemize}
\item[1.] When $\alpha>2$,  $E^*(\omega)$ is strictly positive.
\item[2.] When $\alpha=2$, $E^*(\omega)=0$.
\item[3.] When $0<\alpha<2$,  $E^*(\omega)$ is $-\infty$ for more than 2 bodies.
\end{itemize}

We borrowed the name excited energy from Nakanishi \cite{Na17}, in fact, one may call $E^*(\omega)$ the first excited energy. In principle, we could define different levels of the excited energy $E_j^*(\omega)$ for $\alpha>2$ by induction for $j=1,2,\cdots$,
\begin{equation}
E_j^*(\omega):=\inf\{E_\omega(\bf x): K_\omega(\bf x)=0, E_\omega(\bf x)>E_{j-1}^*(\omega)\},
\end{equation}
where $E_{0}^*(\omega):=0$ and $\inf \emptyset:=\infty$. {Note that when $\alpha=2$, all levels of the exited energy are zero, thus every solution below the excited energy is singular for $\alpha=2$. This represents a very degenerate case in our characterization. In this paper, we will focus on the case where $\alpha>2$.} We only consider solutions with energy below the first excited energy $E_1^*(\omega)$ and we shall denote it by $E^*(\omega)$ and call it the excited energy for simplicity. 

We remark that when defining the excited energy, we choose to fix the frequency $\omega$ instead of fixing the angular momentum. This is because the classical problem of minimum energy configuration with a fixed level of angular momentum is ill-posed for point mass N-body problem (cf. \cite{Sch12}) with $N\geq 3$. In particular, by Sundman's inequality (cf. \cite{MeOf17}), one has \begin{equation}\frac{|A(\bf x, \dot{\bf x})|^2}{2I(\bf x)}+U(\bf x)\leq E(\bf x, \dot{\bf x}).\end{equation}
The minimum energy function for a fixed level of angular momentum $|A(\bf x, \dot{\bf x})|=c$ is defined as  \begin{equation}\mathcal E_c(\bf x):=\frac{c^2}{2I(\bf x)}+U(\bf x),\end{equation}
moreover, the relation between the magnitude of the angular momentum and the frequency of a relative equilibrium with the center of mass at the origin is $c=\omega I(\bf x)$, thus the energy of a relative equilibrium with the magnitude of the angular momentum $c$ and initial configuration $\bf x$ is equal to $\mathcal E_c(\bf x)$. The function $K_\omega(\bf x)$ in terms of $c$ is \[\mathcal K_c(\bf x):=\frac{c^2}{I(\bf x)}+\alpha U(\bf x).\]

If we fix the level of angular momentum as the parameter, the excited energy would be 
\begin{equation}\mathcal E^*(c):=\inf\{\mathcal E_c(\bf x): \mathcal K_c(\bf x)=0\}.\end{equation}
It is easy to check that if $\alpha>2$ and $N\geq 3$, then $\mathcal E^*(c)=0$. That is, the minimum of $\mathcal E_c(\bf x)$ is attained by the ground state where all bodies are at rest at infinity. Moreover, if we define different levels of the excited energy similarly as before, we will get $\mathcal E_j^*(c)=0$ for all $j\geq 1$. Apparently, this is useless if we want to characterize solutions below the excited energy.

Now following ideas from PDE, we consider two sets in the phase space with energy below the excited energy distinguished by the sign of the threshold function:
\begin{equation}\mathcal K^+(\omega)=\{(\bf x, \dot{\bf x}): E(\bf x, \dot{\bf x})<E^*(\omega), K_\omega(\bf x)\geq0\},\end{equation}
\begin{equation}\mathcal K^-(\omega)=\{(\bf x, \dot{\bf x}): E(\bf x, \dot{\bf x})<E^*(\omega), K_\omega(\bf x)<0\}.\end{equation}
We will show that if after some time a solution $\bf x(t)$ remains in $K^+(\omega)$ then it exists globally, and if it remains in $K^-(\omega)$ then it experiences a singularity. Namely,

{ \begin{theorem}[Dichotomy below the excited energy]
\label{thm:dichotomy}
For $\alpha>2$, let $\bf x(t)$ be a solution of the $N$-body problem, if there exists $t^*>0$ so that for $t>t^*$, 
\begin{itemize}
\item[(1)] $\bf x(t)$ stays in $\mathcal K^+(\omega)$, then $\bf x(t)$ exists globally;
\item[(2)] $\bf x(t)$ stays in $\mathcal K^-(\omega)$, then $\bf x(t)$ has a singularity. \end{itemize}
Moreover, every singularity must be collision singularity.
\end{theorem}}

 { When the energy is below zero, the sign of $K_\omega$ either stays negative or there is exactly one transition from positive to negative along each trajectory due to the Lagrange-Jacobi identity. Thus every solution is singular below the ground state energy. When the energy is above zero and below the excited energy, the problem is that the function $K_\omega$ is not sign-definite, and it may change the sign infinitely many times, see the example in section 4. Similar problems occur in PDE when one considers solutions above the ground state and below the first excited state (cf. \cite{Na17}). For the two-body problem the sets $\mathcal K^{\pm}(\omega)$ are invariant by adding a constraint on the angular momentum (Lemma \ref{lem:kep_inv}). And we get the dichotomy for the two-body problem 
\begin{theorem}[Dichotomy for the two-body problem]
\label{thm:dicho_2bd}
For $\alpha>2$, let $m_1+m_2=1$ and $m_1x_1+m_2x_2=0$, \begin{equation}\nonumber \begin{split}
\mathcal K^{+}(\omega)&=\{({\bf x}, \dot{\bf x}): E({\bf x}, \dot{\bf x})<E^*(\omega), |A({\bf x}, \dot{\bf x})|\geq A^*(\omega), K_\omega({\bf x})\geq0\},\\
\mathcal K^{-}(\omega)&=\{({\bf x}, \dot{\bf x}): E({\bf x}, \dot{\bf x})<E^*(\omega), |A({\bf x}, \dot{\bf x})|\geq A^*(\omega), K_\omega({\bf x})<0\}.
\end{split}\end{equation}
then $\mathcal K^{\pm}(\omega)$ are invariant. Here, $E^*(\omega)=m_1m_2\alpha^{\frac{2}{2-\alpha}}(\frac{1}{2}-\frac{1}{\alpha})(\alpha^{\frac{2}{2+\alpha}}\omega^{\frac{\alpha-2}{\alpha+2}})^{\frac{2\alpha}{\alpha-2}}$ and $A^*(\omega)=m_1m_2\alpha^{\frac{2}{2+\alpha}}\omega^{\frac{\alpha-2}{\alpha+2}}$. Solutions in $\mathcal K^{+}(\omega)$ exist globally and solutions in $\mathcal K^{-}(\omega)$ experiences a singularity. 
\end{theorem}
}

For $N\geq 3$, fixing the angular momentum does not guarantee the invariance. We will give some weaker results in terms of the dichotomy of the fates of the solutions for $N\geq 3$. 

{ 
\begin{theorem}[Refinement of the characterization for $N\geq 3$]
\label{thm:ref_nbd}
For $\alpha>2$ and fixed $\omega$, we define \begin{equation}\begin{split}
\mathcal K_1^{+}=\{(\bf x, \dot{\bf x})\in \mathcal K: |A(\bf x, \dot{\bf x})|\geq \omega I(\bf x), K_\omega(\bf x)\geq0\},\\
\mathcal K_1^{-}=\{(\bf x, \dot{\bf x})\in \mathcal K: |A(\bf x, \dot{\bf x})|\geq \omega I(\bf x), K_\omega(\bf x)<0\},\\
\mathcal K_2^{+}=\{(\bf x, \dot{\bf x})\in \mathcal K: |A(\bf x, \dot{\bf x})|< \omega I(\bf x), K_\omega(\bf x)\geq0\},\\
\mathcal K_2^{-}=\{(\bf x, \dot{\bf x})\in \mathcal K: |A(\bf x, \dot{\bf x})|<\omega I(\bf x), K_\omega(\bf x)<0\},\\
\end{split}
\end{equation}
where $\mathcal K=\{(\bf x, \dot{\bf x}): E(\bf x, \dot{\bf x})<E^*(\omega), |A(\bf x, \dot{\bf x})|\neq 0\}=\mathcal K_1^{+}\cup \mathcal K_1^{-}\cup\mathcal K_2^{+}\cup\mathcal K_2^{-}$ is invariant.
\begin{enumerate}
\item[(a)] $\mathcal K_1^+$ is empty.
\item[(b)] If $\bf x(t)$ starts in $\mathcal K_2^-$, and enters $\mathcal K_1^-$, then it stays in $\mathcal K_1^-$ and experiences a collision singularity.
\item[(c)] If $\bf x(t)$ starts in $\mathcal K_2^-$, and never enters $\mathcal K_1^-$, then it stays in $\mathcal K_2^{+}\cup\mathcal K_2^-$. 
\begin{enumerate}
\item[(c1)] If there exists time $t_1$, so that $\bf x(t)$ stays in $\mathcal K_2^-$ after $t_1$, then it experiences a collision;
\item[(c2)] If there exists time $t_1$, so that $\bf x(t)$ stays in $\mathcal K_2^+$ after $t_2$, then it exists globally;
\item[{(c3)}] If there are infinitely many transitions between $\mathcal K_2^+$ and $\mathcal K_2^-$, then it exists globally.
\end{enumerate}
\item[(d)] If $\bf x(t)$ starts in $\mathcal K_2^+ (\textrm{resp.}~ \mathcal K_1^-)$, and stays in $\mathcal K_2^+ (\textrm{resp.}~ \mathcal K_1^-)$, then it  exists globally (resp. experiences a collision).
\item[(e)] If $\bf x(t)$ starts in $\mathcal K_2^+ (\textrm{resp.}~ \mathcal K_1^-)$, and enters $\mathcal K_2^-$, then see (b)(c).
\end{enumerate}
\end{theorem}
}

The paper is organized as follows. In section 2, we study the two-body problem in this new perspective. The characterization of the fates for the two-body problem resembles the results in PDE nicely, and we give a proof for Theorem \ref{thm:dicho_2bd} using the Kepler equation. In section 3 we study the characterization of the fates of the solutions below the excited state energy for $N\geq 3$. More specifically, in section 3.1 we review relative equilibria and central configurations and study their relations to the excited energy. In section 3.2 we give a proof for Theorem \ref{thm:dichotomy}. In 3.3 and 3.4 we add constraints on the angular momentum and refine the characterization of the solutions and give a proof for Theorem \ref{thm:ref_nbd}. In section 4, we provide an example where there are infinitely many transitions for the sign of the function $K_\omega$. In section 5, we give some comments and future plans.

\section{Two-body problem.}

In this section, we study the two-body problem for $\alpha\neq 2$, and one will see the qualitative differences between the cases of $\alpha<2$ and $\alpha>2$. It is well known that the two-body problem can be reduced to the Kepler problem. 
\begin{equation}
\label{eq:Kep}
\ddot{x}=-\nabla U(x),\quad U(x)=-\frac{1}{|x|^\alpha}, \quad \alpha>0, \quad x\in\R^2.
\end{equation}

Here $x$ is the relative position of the two-body, i.e. $x=x_1-x_2$ and we have normalized the total masses to be 1. Since the motion of the two-body problem is always in a plane, we assume $x\in\R^2$ (cf. \cite{MeOf17}). 
We remark that the motion of the two-body problem and the Kepler problem is well-known. What is new here is the characterization of the motion using the idea of excited energy and the function $K_\omega(x)$.

In polar coordinates $(r,\theta)$, the Kepler equation is

\begin{equation}
\begin{split}
\ddot{r}-r\dot{\theta}^2&=-U'(r),\\
\frac{d}{dt}(r^2\dot{\theta})&=0.
\end{split}
\end{equation}
We see the angular momentum is preserved: 
\begin{equation}
x\times \dot{x}=r^2\dot{\theta}=c.
\end{equation}

The total energy is also preserved. 

\begin{equation}
\begin{split}
E(x,\dot{x})&=\frac{1}{2}|\dot{x}|^2+U(x),\\
&=\frac{1}{2}(\dot{r}^2+r^2\dot{\theta}^2)+U(r).\\
\end{split}
\end{equation}

For fixed $c$,
\begin{equation}
\begin{split}
E=\frac{1}{2}\dot{r}^2+\frac{c^2}{2r^2}+U(r).
\end{split}
\end{equation}

Let \begin{equation}V_c(r)=\frac{c^2}{2r^2}+U(r)=\frac{c^2}{2r^2}-\frac{1}{r^\alpha},\end{equation} where $V_c(r)$ is known as the effective potential, we get a one-dimensional conserved system with potential $V_c(r)$. When $c\neq 0$, we have \begin{equation}
\begin{split}
&V_c'(r)=-\frac{c^2}{r^3}+\frac{\alpha}{r^{\alpha+1}}=0,\\
\Rightarrow&r_0=(\frac{c^2}{\alpha})^{\frac{1}{2-\alpha}}\quad \mathrm{is\,\, the\,\, critical\,\, point},\\
\Rightarrow&V_c^*:=V(r_0)=\alpha^{\frac{2}{2-\alpha}}(\frac{1}{2}-\frac{1}{\alpha})c^{\frac{2\alpha}{\alpha-2}}.
\end{split}
\end{equation}
The curves of $V_c(r)$ is ascending when $c\geq0$ is increasing. See Figure \ref{fig:eff_po}. {Note that when $\alpha=2$, the effective potential degenerates to $\frac{c^2/2-1}{2r^2}$, which does not have any critical point, thus no relative equilibrium. }

\begin{figure}
\begin{subfigure}[b]{0.45\textwidth}
\centering
  \includegraphics[width=\textwidth]{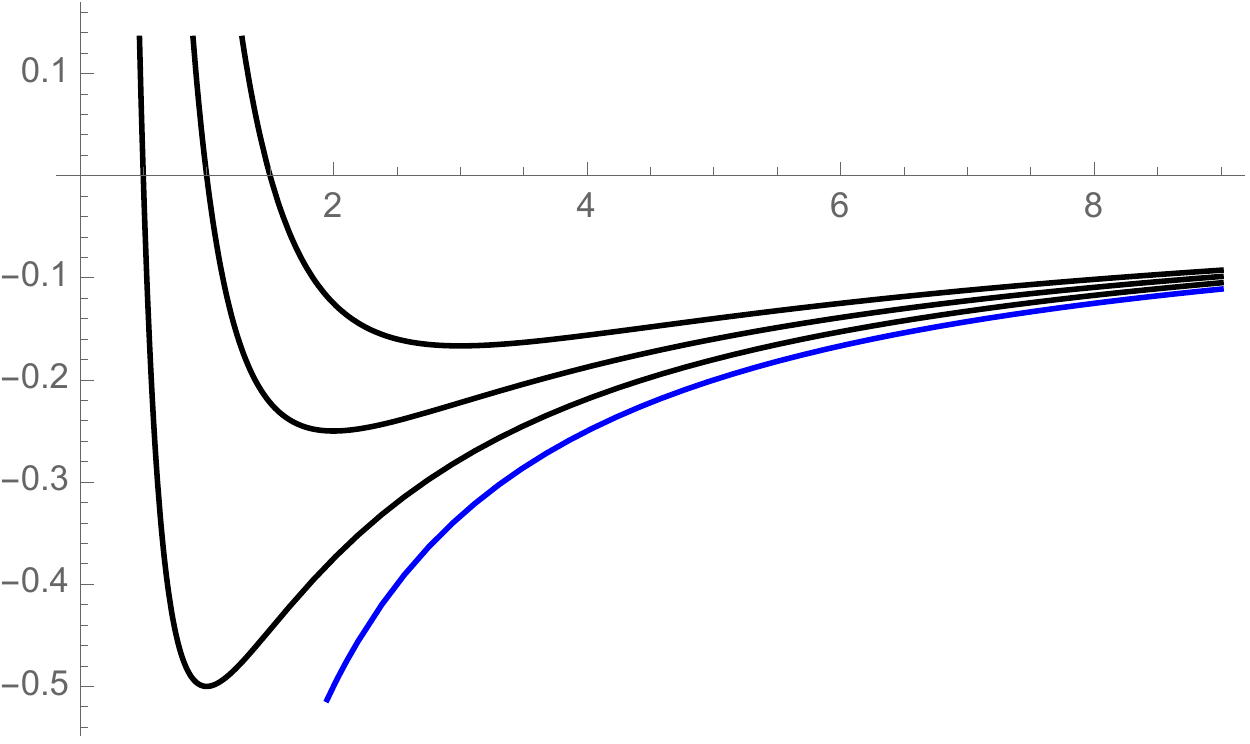}
  \caption{$0<\alpha<2$}
  \end{subfigure}
~\begin{subfigure}[b]{0.45\textwidth}
\centering
  \includegraphics[width=\textwidth]{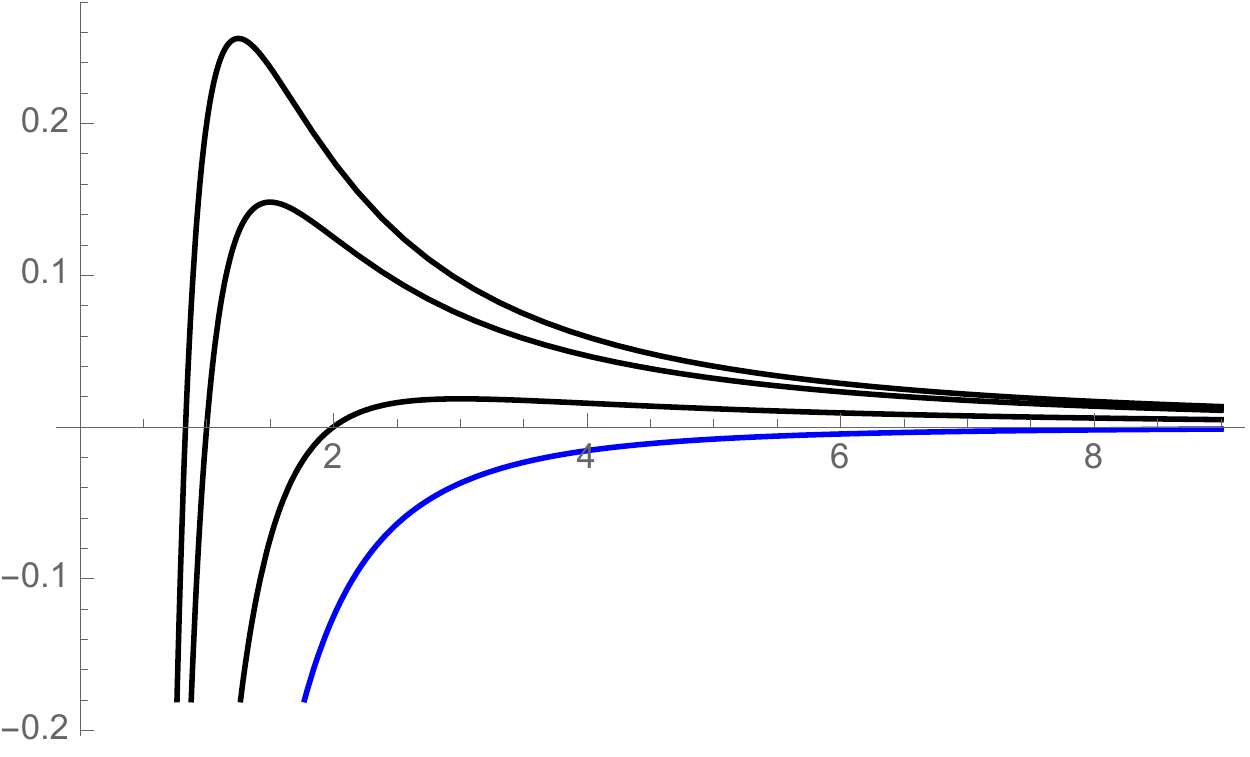}
  \caption{$\alpha>2$}
  \end{subfigure}

\caption{The effective potential $V_c(r)=\frac{c^2}{2r^2}-\frac{1}{r^\alpha}$ for $\alpha\neq 2$, for ascending $c$.}

 \label{fig:eff_po}
\end{figure}

When $\alpha\neq 2$, if $R_\omega (t)\bf q$ is a relative equilibrium (i.e. a circular motion for the Kepler problem), we have $|\bf q|=r_0$ and $r_0^2\dot{\theta}=r_0^2\omega=c$. The relation between $c$ and $\omega$ is given by:
  \begin{equation}\label{c_omega}
\omega=\alpha^{\frac{2}{2-\alpha}}c^{\frac{\alpha+2}{\alpha-2}} \quad \Leftrightarrow\quad c=c(\omega)=\alpha^{\frac{2}{2+\alpha}}\omega^{\frac{\alpha-2}{\alpha+2}}. 
\end{equation}
Note that for the Kepler problem and the two-body problem, there is a unique relative equilibrium for each assigned frequency $\omega$, angular momentum $c$ or radius $r_0$. Any one value of the $\omega, c, r_0$ uniquely determines the values of the other two for a relative equilibrium, and $$K_\omega(\bf q)=\omega^2 |\bf q|^2-\frac{\alpha}{|\bf q|^\alpha}=\omega^2r_0^2-\frac{\alpha}{r_0^\alpha}=0.$$

\begin{figure}[ht]

\centering
  \includegraphics[width=0.6\textwidth]{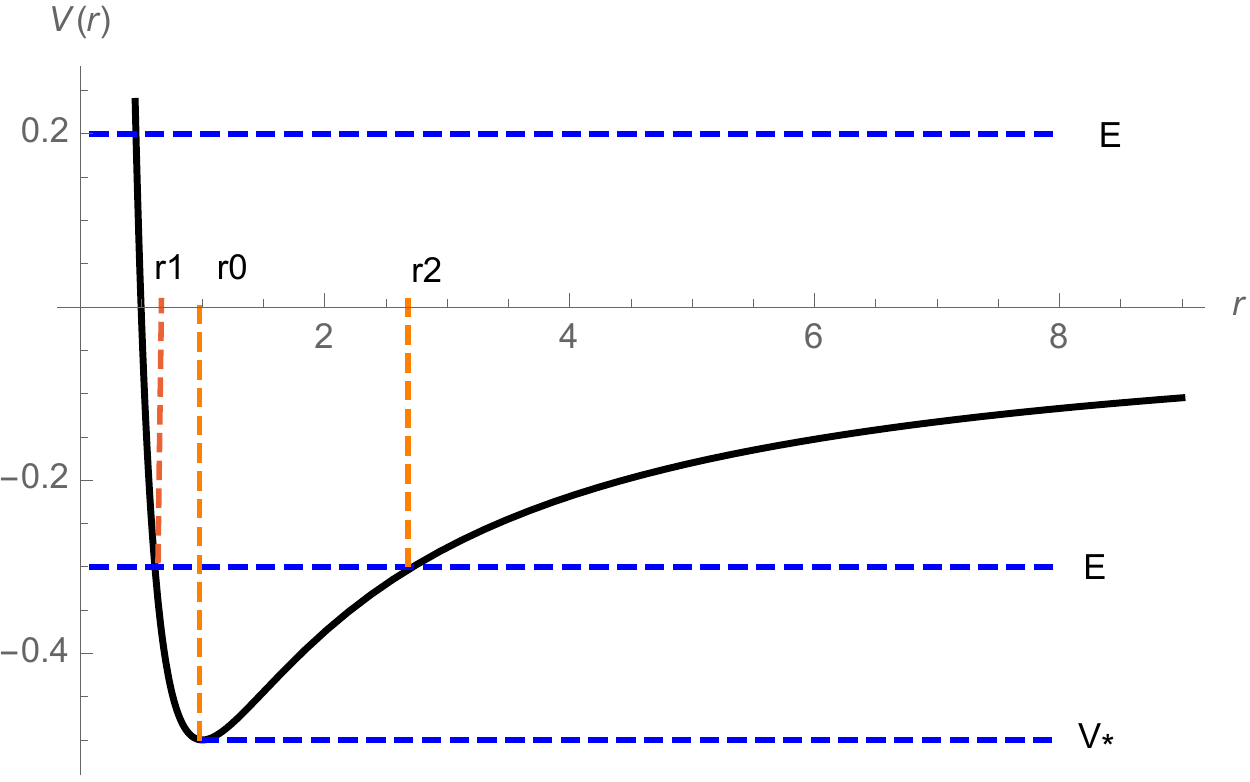}
  \caption{For fixed $c\neq0$ and $0<\alpha<2$, different energy cross sections of the effective potential $V_c(r)=\frac{c^2}{2r^2}-\frac{1}{r^\alpha}$.}

 \label{fig:energyless2}
\end{figure}

\subsection{When $0<\alpha<2$, $c\neq 0$}
The behavior of the orbits resembles the gravitational central force. See Figure \ref{fig:energyless2}.
\begin{enumerate}
\item[1.] When $E=V_c^*$, the orbit is circular with radius $r=r_0$.
\item[2.] When $V_c^*<E<0$, orbits oscillate between $r_1, r_2$ and exist globally.
\item[3.] When $E=0$, the particle barely makes it out to infinity (its speed approaches zero as $r\to\infty$).
\item[4.] When $E>0$, the particle makes it out to infinity with energy to spare.  
\end{enumerate}

\textbf{Conclusion:} No collisions when $c\neq0$, and solution exists for all time. Collision can occur only when $c=0$, i.e. when the particle starts with zero tangential velocity: $\dot{\theta}=0$. Again, we see the set of initial conditions leading to collisions has Lebesgue measure zero for $0<\alpha<2$.

\subsection{When $\alpha>2$, $c\neq 0$}
The effective potential $V_c(r)$ is qualitatively different from that of the gravitational case. See Figure \ref{fig:energybigger2}.
\begin{enumerate}
\item[1.] When $E\leq0$, the orbit will collide at the origin ($r\to 0$). 
\item[2.] When $0<E<V_c^*$, orbits have two cases. If $r<r_1$, it will collide at the origin; if $r>r_2$, the orbit will go to infinity and exist for all time.
\item[3.] When $E=V_c^*$, the orbit will be circular. 
\item[4.] When $E>V_c^*$, the initial position does not give us definite information about the fate of the solution. One also needs the initial radial velocity to determine the fate. See more details at the end of section 2 and Figure \ref{fig:PPKep}. 
\end{enumerate}

\begin{figure}[ht]

\centering
  \includegraphics[width=0.6\textwidth]{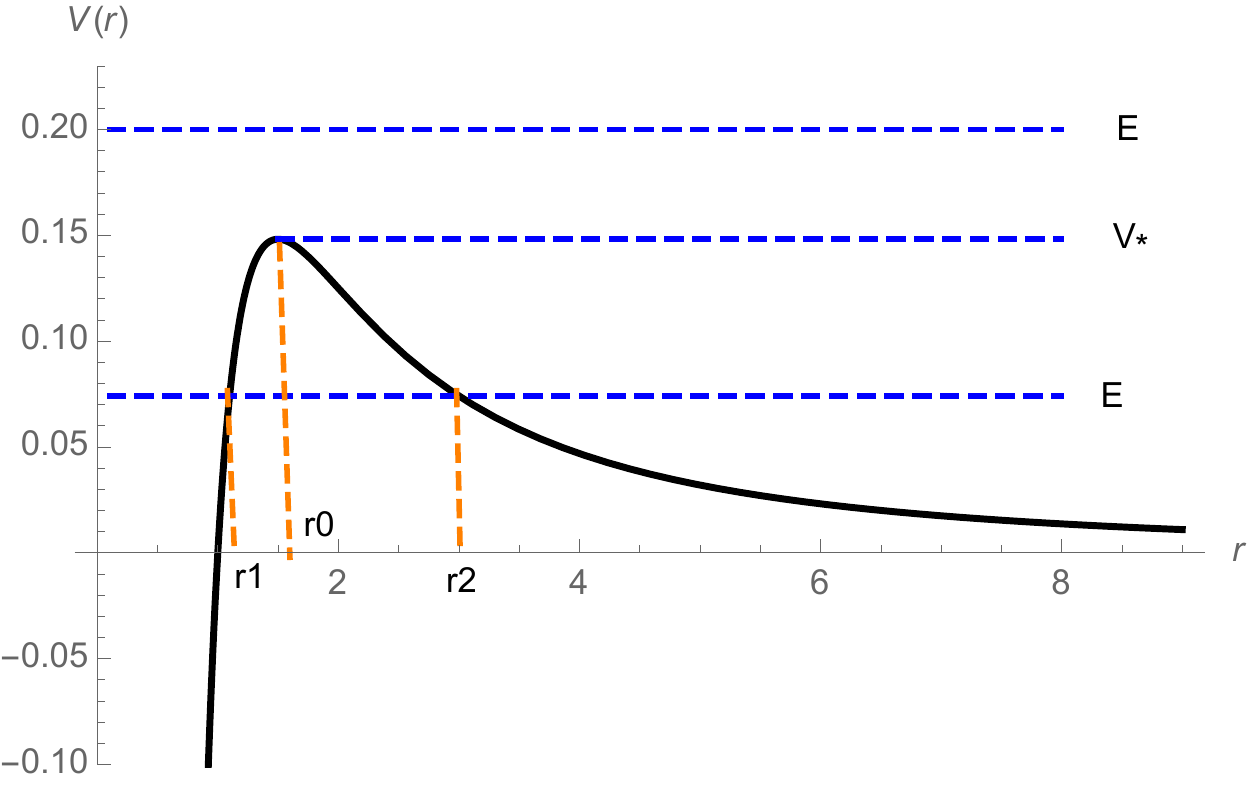}
  \caption{For fixed $c\neq0$ and $\alpha>2$, different energy cross sections of the effective potential $V_c(r)=\frac{c^2}{2r^2}-\frac{1}{r^\alpha}$.}

 \label{fig:energybigger2}
\end{figure}

Let's get back to our attempt of dichotomy between singularity and global existence below the excited energy for $\alpha>2$. If we do not propose conditions on the angular momentum, the constraint $E(x,\dot{x})<V_{c(\omega)}^*$ is not enough to guarantee the invariance of $\mathcal K^{\pm}(\omega)$. Since the curves $V_c(r)$ is ascending as $c$ increases, to make the energy cross section $E$ below the critical value $V_c^*$, we just need to make the angular momentum greater than $c(\omega)$, where $c(\omega)$ is given by (\ref{c_omega}). Then we will get invariant sets.

\begin{lemma}
\label{lem:kep_inv}
Fix $\alpha>2$ and $\omega$. Let \begin{equation}\begin{split}
\mathcal K^{+}(\omega)=\{(x, \dot{x}): E(x, \dot{x})<V_{c(\omega)}^*, x\times \dot{x}\geq c(\omega), K_\omega(x)\geq0\},\\
\mathcal K^{-}(\omega)=\{(x, \dot{x}): E(x, \dot{x})<V_{c(\omega)}^*,x\times \dot{x}\geq c(\omega), K_\omega(x)<0\}.
\end{split}
\end{equation}
then $\mathcal K^{\pm}(\omega)$ are invariant sets for the Kepler problem (\ref{eq:Kep}).
\end{lemma}
\begin{proof}
Since the energy and angular momentum are preserved, $\mathcal K^{+}(\omega)\cup\mathcal K^{-}(\omega)$ is invariant. We only need to show $\mathcal K^{-}(\omega)$ is invariant. Let $x(t)$ be a solution of the Kepler problem with initial conditions in $\mathcal K^{-}(\omega)$. Thus its energy $E$ and angular momentum $c$ satisfies $E<V_{c(\omega)}^*$ and $c\geq c(\omega)$. If there exists time $t_1$ so that $ K_\omega(x(t_1))=0$, i.e. $|x(t_1)|=r_0$. Then $V_c(|x(t_1)|)\geq V_{c(\omega)}(|x(t_1)|)=V_{c(\omega)}^*$. Thus the energy $E(x(t_1), \dot{x}(t_1))\geq V_c(|x(t_1)|)\geq V_{c(\omega)}^*$, contradiction.
\end{proof}

Note that $K_\omega(x)<0$ is equivalent to $|x|<r_0$, and $K_\omega(x)\geq0$ is equivalent to $|x|\geq r_0$.

\begin{proposition}
\label{prop:kep_dicho}
For $\alpha>2$, 
\begin{enumerate}
\item[(1)] Solutions in $\mathcal K^{-}(\omega)$ is singular, i.e. finite time collision.
\item[(2)] Solutions in $\mathcal K^{+}(\omega)$ exist globally.
\end{enumerate}
\end{proposition}
\begin{proof}
Proof of (1). For fixed $\omega$, we will denote $c=c(\omega)$ and $r_0=(\frac{c^2}{\alpha})^{\frac{1}{2-\alpha}}$ is the critical point of $V_c(r)$. Let $x(t)$ be a solution in $\mathcal K^{-}(\omega)$, then there is $\delta>0$ so that $$E(x(t),\dot{x}(t))<V_{c(\omega)}^*-\delta=V_c(r_0)-\delta.$$
Let $I(x)=|x|^2$, then \begin{equation}
\begin{split}
\frac{d^2}{dt^2}I(x(t))&=4[E(x, \dot{x})+(1-\alpha/2)\frac{1}{|x|^\alpha}],\\
&<4(V_c(r_0)+(1-\alpha/2)\frac{1}{r^\alpha})-4\delta.\\
\end{split}
\end{equation}
Let $f(r)=V_c(r_0)+(1-\alpha/2)\frac{1}{r^\alpha}$ and $r<r_0$, easy to check that $f(r)$ is increasing, and $f(r_0)=0$. Thus we have $\ddot{I}(t)<-4\delta$. Thus the time evolution of the moment of inertia (i.e. $I(x)$) is controlled by a concave downward parabola which must become negative for $t\geq t^*$, where $t^*<\infty$. It follows that the particle will collide at the origin in finite time.

Proof of (2). Since $\mathcal K^{+}(\omega)$ is invariant and every solution in it satisfies $|x(t)|>r_0$, thus the solution exists for all time by Painlev{\'e}'s theorem. 
\end{proof}

{Use the elementary relation between the two-body problem and the Kepler problem, more precisely, if $m_1+m_2=1$, $m_1x_1+m_2x_2=0$, then $x_1=m_2x$ and $x_2=-m_1x$. Let $\bf x=(x_1, x_2)$, then 
\begin{equation}
E(\bf x, \dot{\bf x})=\frac{1}{2}\sum_{i=1}^2m_i|\dot{x}_i|^2+U(\bf x)=m_1m_2E(x, \dot{x}),
\end{equation}
\begin{equation}
A(\bf x, \dot{\bf x})=\sum_{i=1}^2 m_i x_i\times \dot{x}_i=m_1m_2x\times\dot{x}.
\end{equation}
From Lemma \ref{lem:kep_inv} and Proposition \ref{prop:kep_dicho}, one can obtain the dichotomy for the two-body problem as in Theorem \ref{thm:dicho_2bd}.}

Moreover, we can get rid of the $\omega$ by taking the union over all $\omega$ of $\mathcal K^{\pm}(\omega)$,\begin{theorem}
\label{thm:dicho_kep}
For $\alpha>2$, let \[\mathcal K^{\pm}=\bigcup_{\omega> 0} \mathcal K^{\pm}(\omega),\]
where $\mathcal K^{\pm}(\omega)$ are given in Lemma \ref{lem:kep_inv} or Theorem \ref{thm:dicho_2bd}, then $\mathcal K^{\pm}$ are invariant. Solutions in $\mathcal K^{+}$ exist globally and solutions in $\mathcal K^{-}$ experiences a singularity.
\end{theorem}

\begin{figure}[ht]

\centering
  \includegraphics[width=0.5\textwidth]{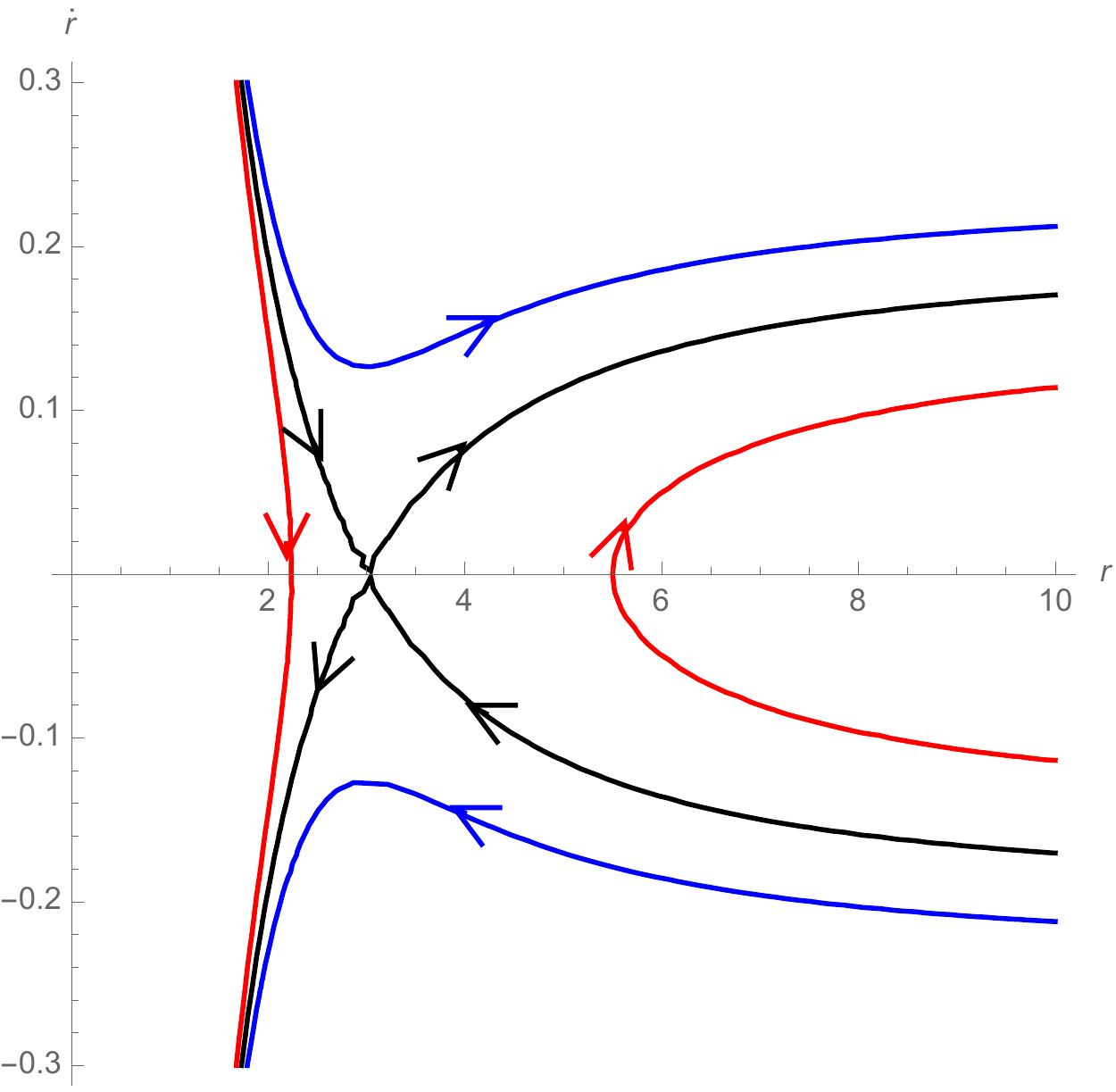}
  \caption{For fixed $c\neq0$ and $\alpha>2$, the phase portrait of the Kepler problem in the $(r,\dot{r})$ plane. Blue curves have energy above the relative equilibrium, and red curves have energy below the relative equilibrium.}

 \label{fig:PPKep}
\end{figure}

Finally, the motions of the Kepler problem are completely predictable. Based on the phase portrait in the $(r, \dot{r})$ plane (Figure \ref{fig:PPKep}), one obtains a trichotomy in forward time:
\begin{itemize}
\item[(i)] finite time collision ($r\to 0$);
\item[(ii)] escaping to infinity ($r\to \infty$);
\item[(iii)] approaching the relative equilibrium.
\end{itemize}
The same holds in backward time, and all nine combinations allowed by the forward/backward time trichotomy can occur, which is similar to the results in PDE as mentioned at the beginning of the paper.

\section{$N$-body problem for the $``\alpha>2"$-potential}
Consider $N$-body ($N\geq 3$) with masses $m_1, \cdots, m_N$ moving in the Euclidean space $\R^3$ under the $\alpha$-potential, where $\alpha>2$. We will fix the center of mass at the origin, i.e. the configuration space is \[X=\{\bf x\in (\R^3)^N\setminus\Delta\, |\quad m_1x_1+\cdots+m_Nx_N=0\}.\]
In this configuration space, the moment of inertia can be expressed in terms of the mutual distances $r_{ij}=|x_i-x_j|$: \[I(\bf x)=\frac{1}{M}\sum\limits_{i<j}m_im_jr_{ij}^2,\]
where $M=m_1+\cdots+m_N$.

\subsection{Relative equilibrium and excited energy}

\begin{definition}[Central configuration, cf. \cite{MeOf17}]
A point $\bf x\in X$ satisfying the equation \[\nabla (\frac{\omega^2}{2}I(\bf x)-U(\bf x))=\bf0,\]
for some $\omega$ is called a central configuration.
\end{definition}

Therefore $\Omega(t)\bf q$ is a relative equilibrium of the $N$-body problem only if $\bf q$ is a central configuration. The reverse is true for $N=3$ but false for $N\geq 4$. The reason is that there are non-planar central configurations when $N\geq 4$, but every relative equilibrium of the N-body problem must be planar:

\begin{proposition}
\label{prop:planarRE}
Relative equilibria of the N-body problem are planar solutions.
\end{proposition}
\begin{proof}
Let $\bf x(t)=\Omega(t)\bf q$ be a relative equilibrium of the $N$-body problem. Without loss of generality, assume $\Omega(t)$ is in the normal form, i.e. the $\omega$-rotation about the $z$-axis. Plug $\bf x(t)$ into the differential equation (\ref{eq:nbody}). The third coordinate of each body's position is a constant $x_{i3}(t)=q_{i3}$, and satisfies \[\sum_{j=1, j\neq i}^{N}\frac{q_{i3}-q_{j3}}{r_{ij}^{\alpha+2}}=0, \quad i=1, \cdots, N.\]  
This is a homogeneous linear system for $(q_{13}, \cdots, q_{N3})$ and the coefficient matrix is 
\[C=\begin{pmatrix}\sum_{j=2}^Na_{1j}&-a_{12}&\cdots&-a_{1N} \\ -a_{21}&\sum_{j=1, j\neq 2}^Na_{2j}&\cdots&-a_{2N}\\&&\ddots&\\-a_{N1}&-a_{N2}&\cdots&\sum_{j=1}^{N-1}a_{Nj}\end{pmatrix},\]
where $a_{ij}=a_{ji}=1/r_{ij}^{\alpha+2}$ and $r_{ij}=|q_i-q_j|$. The Kernel of $C$ is $Span\{(1,1, \cdots, 1)\}$. Thus $q_{13}=q_{23}=\cdots=q_{N3}$ and the motion stays in a plane orthogonal to the $z$-axis. 
\end{proof}

\begin{lemma}[The sign of the excited energy]
\label{lem:ground}
Fix $\alpha>0$ and $\omega$, 
\begin{equation}\nonumber E^*(\omega)=\inf\{E_\omega(\bf x): K_\omega(\bf x)=0\}.\end{equation}
\begin{itemize}
\item[1.] When $\alpha>2$,  $E^*(\omega)$ is strictly positive.
\item[2.] When $\alpha=2$, $E^*(\omega)=0$.
\item[3.] When $0<\alpha<2$,  $E^*(\omega)$ is $-\infty$ for $N\geq 3$.
\end{itemize}
\end{lemma}

\begin{proof}
Note that when $K_\omega(\bf x)=0$, we have 
\begin{equation}E_\omega(\bf x)=-(\frac{\alpha}{2}-1)U(\bf x).\end{equation}
Let $U^*(\omega)=\inf\{-U(\bf x): K_\omega(\bf x)=0\}$, then $E^*(\omega)=(\frac{\alpha}{2}-1)U^*(\omega)$ for $\alpha\geq 2$. When $\alpha=2$, $E^*(\omega)=0$ is trivial. We only need to show $U^*(\omega)$ is strictly positive. 

In terms of mutual distances, \[K_\omega(\bf x)=\omega^2 I(\bf x)+\alpha U(\bf x)=\frac{\omega^2}{M}\sum\limits_{i<j}m_im_jr_{ij}^2-\alpha\sum\limits_{i<j}\frac{m_im_j}{r_{ij}^\alpha}.\]
Under the constraint $K_\omega(\bf x)=0$, $U^*(\omega)$ cannot be zero. Moreover, the infimum of $-U$ can be achieved in the set $\{\bf x\in X: K_\omega(\bf x)=0\}$.

When $\alpha<2$, $E^*(\omega)=(\frac{\alpha}{2}-1)\sup\{-U(\bf x): K_\omega(\bf x)=0\}$. The supreme of $-U(\bf x)$ with $K_{\omega}(\bf x)=0$ is infinity for $N\geq 3$. For example, when $N=3$ one can find a sequence of $\bf r_n=(r_{12}^{(n)}, r_{13}^{(n)}, r_{23}^{(n)})$ so that $K_\omega(\bf r_n)=0$ for all $n$ and $\lim_{n\to \infty}r_{12}^{(n)}=0$, $\lim_{n\to \infty}r_{23}^{(n)}=\infty$, and $\lim_{n\to \infty}r_{13}^{(n)}=\infty$. Similarly for any $N>3$.
\end{proof}

\begin{proposition}[Excited energy and central configuration]
\label{prop:EEcc}
When $\alpha>2$, the excited energy $E^*(\omega)$ is attained by a central configuration.
\end{proposition}
\begin{proof}
When $\alpha>2$, we have \[E^*(\omega)=(\alpha/2-1)\inf \{-U(\bf x): K_\omega(\bf x)=0\},\]
and we know the infimum is strictly positive and achieved by some point $\bf q$ satisfying $K_\omega(\bf q)=0$. By Lagrange multipliers of constrained optimization, we know there is some $\lambda\in\R$ such that
\[-\nabla U(\bf q)=\lambda\nabla K_{\omega}(\bf q)=\lambda(\omega^2\nabla I(\bf q)+\alpha\nabla U(\bf q)).\]
Take inner product with $\bf q$ on both sides, we get
\begin{equation}
-\bf q\cdot \nabla U=\alpha U=\lambda(2\omega^2 I-\alpha^2 U)=-\lambda(2\alpha U+\alpha^2 U),
\end{equation}
thus $\lambda=-\frac{1}{2+\alpha}$. Therefore,
 \begin{equation}
-\nabla U(\bf q)=-\frac{1}{2+\alpha}(\omega^2\nabla I(\bf q)+\alpha\nabla U(\bf q)),
\end{equation}
which implies \[\omega^2\nabla I(\bf q)-2\nabla U(\bf q)=0.\]
thus $\bf q$ is a central configuration.
\end{proof}

We call a relative equilibrium an \emph{excited state} if its energy is equal to $E^*(\omega)$ for the corresponding frequency $\omega$. The question about the excited states and different levels of the excited states is not important in the current paper, because we only consider energy below the first excited energy. What matters is the positivity of the excited energy $E^*(\omega)$. In a subsequent work, we will investigate the excited states.

\subsection{A preliminary dichotomy below the excited energy.}
As been mentioned in the introduction, we let
\[\mathcal K^+(\omega)=\{(\bf x, \dot{\bf x}): E(\bf x, \dot{\bf x})<E^*(\omega), K_\omega(\bf x)\geq0\},\]
\[\mathcal K^-(\omega)=\{(\bf x, \dot{\bf x}): E(\bf x, \dot{\bf x})<E^*(\omega), K_\omega(\bf x)<0\}.\]

We give a proof of Theorem \ref{thm:dichotomy} in this section. First, we present two lemmas.

\begin{lemma}
\label{lem:intermediate}
Let $U^*(\omega):=\inf\{-U(\bf x): K_\omega(\bf x)=\omega^2 I(\bf x)+\alpha U(\bf x)=0\}$. If $K_\omega(\bf x)<0$, then $-U(\bf x)> U^*(\omega)$.
\end{lemma}
\begin{proof}
For fixed $\bf x$ with $K_\omega(\bf x)<0$, let $f(\lambda)=K_\omega(\lambda\bf x)$. By the homogeneity of $I, U$ we have \begin{equation}\nonumber K_\omega(\lambda\bf x)=\omega^2\lambda^2 I(\bf x)+(\lambda)^{-\alpha}\alpha U(\bf x).\end{equation}
Therefore,
\begin{itemize}
\item $\lambda=1$, $f(1)=K_\omega(\bf x)<0$,
\item $\lambda\to \infty$, $f(\lambda)>0$.
\end{itemize}
Thus there exists $\lambda^*>1$ so that $K_\omega(\lambda^*\bf x)=0$, therefore
\begin{equation}
\nonumber
\begin{split}
U^*(\omega)&\leq-U(\lambda^*\bf x)=-(\lambda^*)^{-\alpha}U(\bf x),\\
-U(\bf x)&\geq (\lambda^*)^{\alpha}U^*(\omega)>U^*(\omega).
\end{split}
\end{equation}
\end{proof}

\begin{lemma}
\label{lem:intermediate2}
\begin{equation}
\nonumber
\begin{split}E^*(\omega)\leq\inf\{\frac{\omega^2}{2} I(\bf x)+U(\bf x): K_\omega(\bf x)>0\}.\end{split}\end{equation}
 \end{lemma}
\begin{proof}Fix $\bf x$ with $K_\omega(\bf x)>0$, let $f(\lambda)=K_\omega(\lambda\bf x)$. By the homogeneity of $I, U$: \begin{equation} K_\omega(\lambda\bf x)=\omega^2\lambda^2 I(\bf x)+(\lambda)^{-\alpha}\alpha U(\bf x).\nonumber\end{equation}
Therefore,
\begin{itemize}
\item $\lambda=1$, $f(1)=K_\omega(\bf x)>0$,
\item $\lambda\to 0$, $f(\lambda)\to -\infty$.
\end{itemize}
Thus there exists $0<\lambda^*<1$ so that $K_\omega(\lambda^*\bf x)=0$. Therefore
\begin{equation}
\begin{split}
\frac{\omega^2}{2} I(\bf x)+U(\bf x)&>(\lambda^*)^2\frac{\omega^2}{2} I(\bf x)+(\lambda^*)^{-\alpha}U(\bf x),\\
&=\frac{\omega^2}{2} I(\lambda^*\bf x)+U(\lambda^*\bf x),\\
&\geq E^*(\omega).
\end{split}
\end{equation}
\end{proof}

{Now we are ready to give a proof of Theorem \ref{thm:dichotomy}.}

\begin{proof}[Proof of Theorem \ref{thm:dichotomy}]
Proof of (1). Suppose $\bf x(t)$ does not exist globally, by Painlev{\'e}'s theorem, there exists $\sigma>0$ so that \[\lim_{t\to \sigma}\min_{i\neq j} r_{ij}=0.\]
Since $\bf x(t)$ stays in $\mathcal K^+(\omega)$, $\omega^2 I(t)+\alpha U(t)\geq 0$ for $t>t^*$, we have \begin{equation}
I(t)\to\infty, U(t)\to-\infty \quad \mathrm{as} \quad t\to \sigma. \nonumber
\end{equation}
On the other hand, $\ddot{I}=4[E(\bf x, \dot{\bf x})+(\alpha/2-1)U(\bf x)]$, we have $\ddot{I}(t)\to -\infty$ as $t\to \sigma$ and this is a contradiction to $I(t)\to\infty$ as $t\to \sigma$. Thus solutions in $\mathcal K^+(\omega)$ must exist globally.

Proof of (2). Let $E(\bf x(t),\dot{\bf x}(t))=E^*(\omega)-\delta$. When $t>t^*$, we have $K_\omega(\bf x(t))<0$.
 \begin{equation}
 \nonumber
 \begin{split}
 \ddot{I}(\bf x(t))&=4[E^*(\omega)-\delta+(\alpha/2-1)U(\bf x(t))],\\
 &=4[(\frac{\alpha}{2}-1)U^*(\omega)-\delta+(\alpha/2-1)U(\bf x(t))],\\
 &=4(\frac{\alpha}{2}-1)[U^*(\omega)+U(\bf x(t))]-4\delta,\\
 &<-4\delta.
 \end{split}
 \end{equation}
 The last inequality is by Lemma \ref{lem:intermediate}. Thus the time evolution of the moment of inertia is controlled by a concave downward parabola which must become negative for $t\geq t_1$, where $t_1<\infty$. It follows that the solution must have a singularity.

Furthermore, Von Zeipel's Theorem tells us that if $\sigma$ is a non-collision singularity then
\[\lim_{t\to \sigma} I(t)=+\infty,\] and this cannot happen when $\alpha>2$ as seen in the proof of Theorem of (1), thus singularities of the N-body problem for $\alpha>2$ must be collision singularities.  
 
\end{proof}

As we have pointed out in the introduction, $\mathcal K^\pm(\omega)$ are not invariant sets. We show this by two simple examples for $N=3$.

\begin{example}[Example for the non-invariance of $\mathcal K^+(\omega)$]
\begin{equation}\nonumber \mathcal K^+(\omega)=\{(\bf x, \dot{\bf x}): E(\bf x, \dot{\bf x})<E^*(\omega), K_\omega(\bf x)\geq0\},\end{equation}

\begin{equation}\nonumber K_\omega(\bf x)=\frac{\omega^2}{M}\sum\limits_{i<j}m_im_jr_{ij}^2-\alpha\sum\limits_{i<j}\frac{m_im_j}{r_{ij}^\alpha}.\end{equation}
Let's start with an equilateral triangle configuration $\bf x^0$ and initial velocity $\dot{\bf x}^0=\bf 0$. As long as $(\sqrt{3}|x_i^0|)^{2+\alpha}\geq \frac{\alpha M}{\omega^2}$ for $i=1,2,3$, then $(\bf x^0,\bf 0)\in\mathcal K^+(\omega)$. 

By the attracting forces of the 3 bodies, all of which point to the center of mass (the origin), the 3 bodies will encounter a total collision in finite time. This corresponds to a homothetic motion \cite{Wintner}. Clearly, the solution $(\bf x(t), \dot{\bf x}(t))$ with initial condition $(\bf x^0,\bf 0)\in\mathcal K^+(\omega)$ will enter $\mathcal K^-(\omega)$ after some time $t_1$ and stays in $\mathcal K^-(\omega)$. Thus $\mathcal K^+(\omega)$ is not invariant under the flow.
\end{example}

\begin{example}[Example for the non-invariance of $\mathcal K^-(\omega)$]
Similarly, let's start with an equilateral triangle configuration $\bf x^0$ and initial velocity $\dot{\bf x}^0=v \bf x^0$, where $v>0$. We can choose $(\bf x^0,  \dot{\bf x}^0)\in \mathcal K^-(\omega)$ and $E(\bf x^0, \dot{\bf x}^0)>0$. Since \begin{equation}E(\bf x, \dot{\bf x})=\frac{1}{2}\sum_{i=1}^3m_i|\dot{x}_i|^2+U(\bf x),\end{equation}
is conserved and $U(\bf x)<0$, the three bodies will keep going away ($|\dot{\bf x}|\neq 0$) and never come back, thus enter the set $\mathcal K^+(\omega)$.
\end{example}

In section 4, we will provide an example where there are infinitely many transitions between $\mathcal K^+(\omega)$ and $\mathcal K^-(\omega)$. Now we would like to refine the characterization by making use of the conservation of the angular momentum.

\subsection{Angular momentum and rotating coordinates.} The angular momentum of the N-body system is another important integral besides the total energy. Recall the angular momentum is \[A(\bf x, \dot{\bf x})=\sum_{i=1}^N m_i x_i\times \dot{x}_i.\]
It is a constant vector in $\R^3$ under the motion. To make use of the angular momentum, we first present some results about the rotational coordinates. Let's take the uniform rotating coordinates $$\bf x=\exp(\omega Jt)\tilde{\bf x},$$
where \[J=\begin{pmatrix}0&-1&0\\1&0&0\\0&0&0\end{pmatrix}, \quad \exp(\omega Jt)=\begin{pmatrix}\cos(\omega t)&-\sin(\omega t)&0\\ \sin(\omega t)&\cos(\omega t)&0\\0&0&1\end{pmatrix}.\]
The differential equations for the $N$-body problem in the uniform rotating coordinates is:
\begin{equation}
\nonumber
m_i(\ddot{\tilde{x}}_i+2\omega J\dot{\tilde{x}}_i)=-\omega^2m_i J^2\tilde{x}_i-\nabla_iU(\tilde{\bf x})=\nabla_i(\frac{\omega^2}{2}\sum_{i=1}^Nm_i(\tilde{x}_{i1}^2+\tilde{x}_{i2}^2)-U(\tilde{\bf x})),
\end{equation}
where $i=1,\cdots, N$.

The energy is 
\begin{equation}
\begin{split}
E(\bf x, \dot{\bf x})&=\frac{1}{2}\sum_{i=1}^N m_i|\dot{x}_i|^2+U(\bf x),\\
&=\frac{1}{2}\sum_{i=1}^N m_i |\omega J\tilde{x}_i+\dot{\tilde{x}}_i|^2+U(\bf x),\\
&=\frac{\omega^2}{2}\sum_{i=1}^Nm_i(\tilde{x}_{i1}^2+\tilde{x}_{i2}^2)-\omega \sum_{i=1}^N m_i \tilde{x}_i^TJ\dot{\tilde{x}}_i+\frac{1}{2}\sum_{i=1}^N m_i|\dot{\tilde{x}}_i|^2+U(\bf x).
\end{split}
\end{equation}

The angular momentum is 
\begin{equation}
\begin{split}
A(\bf x, \dot{\bf x})&=\sum_{i=1}^N m_i x_i\times \dot{x}_i,\\
&=\exp(\omega Jt)[\sum_{i=1}^Nm_i\tilde{x}_i\times (\omega J\tilde{x}_i)+\sum_{i=1}^N m_i \tilde{x}_i\times\dot{\tilde{x}}_i].\\
\end{split}
\end{equation}
In particular, the third coordinate of $A(\bf x, \dot{\bf x})$ is \[(A(\bf x, \dot{\bf x}))_3=\omega\sum_{i=1}^Nm_i(\tilde{x}_{i1}^2+\tilde{x}_{i2}^2)+\sum_{i=1}^N m_i (\tilde{x}_i\times\dot{\tilde{x}}_i)_3.\]

Elementary calculation shows that \begin{equation}\sum_{i=1}^N m_i \tilde{x}_i^TJ\dot{\tilde{x}}_i=-\sum_{i=1}^N m_i (\tilde{x}_i\times\dot{\tilde{x}}_i)_3.\end{equation}

Therefore the energy can be written as
\begin{equation}
\label{eq:er_ang}
E(\bf x, \dot{\bf x})=-\frac{\omega^2}{2}\sum_{i=1}^Nm_i(x_{i1}^2+x_{i2}^2)+U(\bf x)+\omega (A(\bf x, \dot{\bf x}))_3+\frac{1}{2}\sum_{i=1}^N m_i|\dot{\tilde{x}}_i|^2.
\end{equation}

\begin{lemma}
\label{lem:inv3bd}
Fix $\alpha>2$ and $\omega$. Let $\bf x(t)$ be a solution with energy $E<E^*(\omega)$. If there exists time $t_1$ so that $K_\omega(\bf x(t_1))=0$, then $|A|<\omega I(t_1)$, where $|A|$ is the magnitude of the angular momentum of $\bf x(t)$.
\end{lemma}
\begin{proof} Let $A(\bf x(t), \dot{\bf x}(t))=\bf a \in \R^3$, then there exists $P\in SO(3)$ so that $P\bf a=(0,0,|A|)$. Since $P\bf x(t)$ is also a solution and its angular momentum is $(0,0,|A|)$, without loss of generality, we may assume the angular momentum of $\bf x(t)$ is $(0,0,|A|)$.
For the sake of contradiction, let's suppose $|A|\geq \omega I(t_1)$, by equation (\ref{eq:er_ang}) we have
\begin{equation}
\label{eq:inv3bd}
\begin{split}
E&(\bf x(t_1), \dot{\bf x}(t_1))\\&=-\frac{\omega^2}{2}\sum_{i=1}^Nm_i(x_{i1}^2(t_1)+x_{i2}^2(t_1))+U(\bf x(t_1))+\omega |A|+\frac{1}{2}\sum_{i=1}^N m_i|\dot{\tilde{x}}_i|^2,\\
&\geq-\frac{\omega^2}{2}\sum_{i=1}^Nm_i(x_{i1}^2(t_1)+x_{i2}^2(t_1))+U(\bf x(t_1))+\omega^2I(\bf x(t_1)),\\
&\geq \frac{\omega^2}{2}I(\bf x(t_1))+U(\bf x(t_1))\geq E^*(\omega).
\end{split}
\end{equation}
this is a contradiction to the assumption $E<E^*(\omega)$. 
\end{proof}

\subsection{Four subsets with energy below the excited energy.}

We define \begin{equation}\begin{split}
\mathcal K_1^{+}=\{(\bf x, \dot{\bf x})\in \mathcal K: |A(\bf x, \dot{\bf x})|\geq \omega I(\bf x), K_\omega(\bf x)\geq0\},\\
\mathcal K_1^{-}=\{(\bf x, \dot{\bf x})\in \mathcal K: |A(\bf x, \dot{\bf x})|\geq \omega I(\bf x), K_\omega(\bf x)<0\},\\
\mathcal K_2^{+}=\{(\bf x, \dot{\bf x})\in \mathcal K: |A(\bf x, \dot{\bf x})|< \omega I(\bf x), K_\omega(\bf x)\geq0\},\\
\mathcal K_2^{-}=\{(\bf x, \dot{\bf x})\in \mathcal K: |A(\bf x, \dot{\bf x})|<\omega I(\bf x), K_\omega(\bf x)<0\},\\
\end{split}
\end{equation}
where $\mathcal K=\{(\bf x, \dot{\bf x}): E(\bf x, \dot{\bf x})<E^*(\omega), |A(\bf x, \dot{\bf x})|\neq 0\}=\mathcal K_1^{+}\cup \mathcal K_1^{-}\cup\mathcal K_2^{+}\cup\mathcal K_2^{-}$ is invariant. Note that $\mathcal K_{1,2}^{\pm}$ depends on $\omega$, for notational simplicity we omit the $\omega$ when there is no confusion.

\begin{lemma}
\label{lem:4sets}
The set $\mathcal K_1^+$ is empty. The set $\mathcal K_2^-$ can go to either $\mathcal K_1^-$ or $\mathcal K_2^+$. The set $\mathcal K_1^-$ can only go to $\mathcal K_2^-$, and $\mathcal K_2^+$ can only go to $\mathcal K_2^-$. See figure \ref{fig:4sets}. \end{lemma} 
\begin{proof}

\begin{figure}
	\centering
	\begin{tikzpicture}
     \draw [->] (-3, 0)--(3, 0)node at (3,0) [right] {$K_\omega$};
	\draw [->](0, -3)--(0, 3)node at (0,3) [above] {$|A|-\omega I$};
	\draw (0.2,0) node [above] {$0$};
	\draw (1.5,1.5) node  {$\mathcal K_1^+$};
        \draw (-1.5,1.5) node  {$\mathcal K_1^-$};
	\draw (-1.5,-1.5) node  {$\mathcal K_2^-$};
	\draw (1.5,-1.5) node  {$\mathcal K_2^+$};
	\draw [thick, red]  (1.5,1.5) circle (0.8);
	\draw [thick, red, ->] (-2, 1.3)--(-2, -1.3);
	\draw [thick, red, ->] (1.3, -2)--(-1.3, -2);
	\draw [thick, blue, ->] (-1.1, -1)--(-1.1, 1.1);
	\draw [thick, blue, ->] (-1, -1.1)--(1.1, -1.1);
	\end{tikzpicture}
	\caption{$\mathcal K_1^+$ is empty, $\mathcal K_1^-$ can go to $\mathcal K_2^-$, and $\mathcal K_2^+$ can go to $\mathcal K_2^-$. The set $\mathcal K_2^-$ can go to either $\mathcal K_1^-$ or $\mathcal K_2^+$.} \label{fig:4sets}
\end{figure}
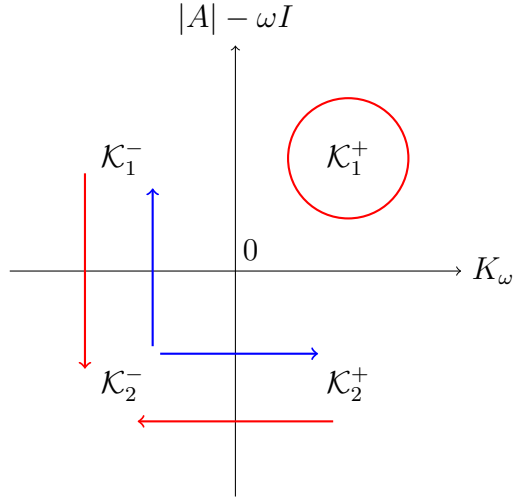

Let $\bf x(t)$ be a solution of the N-body problem in $\mathcal K$. Let $A(t)\equiv A(0), I(t), K_\omega(t)$ be the values of $A(\bf x, \dot{\bf x})$, $I(\bf x), K_\omega(\bf x)$ along the solution at time $t$. From Lemma \ref{lem:inv3bd} we know $|A(t)|=\omega I(t)$ and $K_\omega(t)=0$ cannot happen simultaneously. We study all the possible transitions among these four sets along $\bf x(t)$. 
\begin{enumerate}
\item[1.] Start in $\mathcal K_1^+$. To leave $\mathcal K_1^+$ means there is time $t_1>0$ so that 
\begin{equation}
|A(t_1)|=\omega I(t_1), K_\omega(t_1)>0, \quad (i) \nonumber
\end{equation}
or \begin{equation}
|A(t_1)|>\omega I(t_1), K_\omega(t_1)=0. \quad (ii) \nonumber
\end{equation}
Case (i) is not possible because of Lemma \ref{lem:intermediate2} and an obvious modification of (\ref{eq:inv3bd}). 
Case (ii) is not possible by Lemma \ref{lem:inv3bd}. Thus $\mathcal K_1^+$ must be invariant under the flow of the N-body problem. Moreover, suppose $\bf x(t)$ is a solution in $\mathcal K_1^+$, we know $|A(t)|\geq\omega I(t), K_\omega(t)\geq0$ cannot happen simultaneously similar to the reasoning of cases (i)(ii). Therefore, $\mathcal K_1^+$ must be an empty set.

\item[2.] Start in $\mathcal K_1^-$. To leave $\mathcal K_1^-$ means there is time $t_1>0$ so that 
\begin{equation}
|A(t_1)|=\omega I(t_1), K_\omega(t_1)<0, \quad (iii) \nonumber
\end{equation}
or \begin{equation}
|A(t_1)|>\omega I(t_1), K_\omega(t_1)=0. \quad (ii) \nonumber
\end{equation}
Case (ii) is not possible as we have seen, and case (iii) is possible. So $\mathcal K_1^-$ can go to $\mathcal K_2^-$.

\item[3.] Start in $\mathcal K_2^+$. To leave $\mathcal K_2^+$ means there is time $t_1>0$ so that 
\begin{equation}
|A(t_1)|=\omega I(t_1), K_\omega(t_1)>0, \quad (i) \nonumber
\end{equation}
or \begin{equation}
|A(t_1)|<\omega I(t_1), K_\omega(t_1)=0. \quad (iv) \nonumber
\end{equation}
Case (i) is not possible as we have seen, and case (iv) is possible. So $\mathcal K_2^+$ can go to $\mathcal K_2^-$.

\item[4.] Start in $\mathcal K_2^-$. To leave $\mathcal K_2^-$ means there is time $t_1>0$ so that 
\begin{equation}
|A(t_1)|=\omega I(t_1), K_\omega(t_1)<0, \quad (iii) \nonumber
\end{equation}
or \begin{equation}
|A(t_1)|<\omega I(t_1), K_\omega(t_1)=0. \quad (iv) \nonumber
\end{equation}
Both case (iii) and case (iv) are possible. So $\mathcal K_2^-$ can go to $\mathcal K_1^-$ or $\mathcal K_2^+$. 
\end{enumerate}
\end{proof}

Now we only need to characterize solutions in the set $\mathcal K\setminus \mathcal K_1^+$. Let $\bf x(t)$ be a solution in $\mathcal K\setminus \mathcal K_1^+$, and $\delta=E^*(\omega)-E(\bf x(t),\dot{\bf x}(t))$. Note that whenever $K_\omega(\bf x(t))<0$, $\ddot{I}(t)\leq -4\delta$.

\begin{lemma}
\label{lem:k1minus}
Suppose $\bf x(t)$ starts in $\mathcal K_2^{+}\cup\mathcal K_2^-$, if there exists $t_1$ so that $\bf x(t_1)\in\mathcal K_1^-$ then $\bf x(t)$ remains in $\mathcal K_1^-$ for all $t>t_1$.
\end{lemma}
\begin{proof}
Without loss, we assume $t_1$ is the first time that $I(t_1)=|A|/\omega$. Since $\bf x(t)$ can only go into $\mathcal K_1^-$ from $\mathcal K_2^-$, and $I(t)$ in $\mathcal K_2^-$ is concave downward and greater than $|A|/\omega$, we have $\dot{I}(t_1)<0$. So for $t>t_1$ and close to $t_1$ we have $I(t)<|A|/\omega$ and $K_\omega(t)<0$. Thus $\ddot{I}(t)\leq -4\delta$ and $I(t)$ is concave downward, thus $I(t)$ remains less than $|A|/\omega$, i.e. the solution remains in $\mathcal K_1^-$ for all $t>t_1$.
\end{proof}

\begin{lemma}
\label{lem:k1minus2}
Suppose $\bf x(t)$ starts in $\mathcal K_1^-$, and $\dot{I}(0)\leq 0$, then $\bf x(t)$ remains in $\mathcal K_1^-$, and $\bf x(t)$ must have a singularity.
\end{lemma}
\begin{proof}
For $t=0$, we have $\ddot{I}(0)\leq -4\delta$ and $\dot I(0)\leq 0$,  thus $I(t)$ remains less than $|A|/\omega$, i.e. the solution remains in $\mathcal K_1^-$ for all $t>0$. By Theorem \ref{thm:dichotomy}, $\bf x(t)$ is singular.
\end{proof}

{ In conclusion, we get the characterization of solutions as in Theorem \ref{thm:ref_nbd}.}

\section{Infinitely many transitions between $\mathcal K^+$ and $\mathcal K^-$}
Recall
\[\mathcal K^+=\{(\bf x, \dot{\bf x}): E(\bf x, \dot{\bf x})<E^*(\omega), K_\omega(\bf x)\geq0\},\]
\[\mathcal K^-=\{(\bf x, \dot{\bf x}): E(\bf x, \dot{\bf x})<E^*(\omega), K_\omega(\bf x)<0\}.\]

From the previous discussions, we know the major difficulty in characterizing solutions below the excited energy is the non-invariance of the sets $\mathcal K^{\pm}$. In particular, the possibility of infinitely many transitions between $\mathcal K^\pm$ complicates the problem. In this section, we provide an example to verify that infinitely many transitions between $\mathcal K^+$ and $\mathcal K^-$ exist, indicating that the characterization of solutions for the N-body problem in this new perspective is challenging as well.

The threshold function $K_\omega(\bf x)$ is
\begin{equation}\nonumber\omega^2 I(\bf x)+\alpha U(\bf x)=\frac{\omega^2}{M}\sum\limits_{i<j}m_im_jr_{ij}^2-\alpha\sum\limits_{i<j}\frac{m_im_j}{r_{ij}^\alpha}.\end{equation}
When all the mutual distances $r_{ij}$ are ``large", $K_\omega>0$; and when $r_{ij}$ are ``small", $K_\omega<0$. A periodic or quasi-periodic solution whose mutual distances oscillate between ``large" and ``small" would suffice. Such solutions are common in the Newtonian ($\alpha=1$) 3-body problem, for example, the elliptic Lagrange homographic solutions. The configuration remains similar (equilateral triangle) and all three masses move along elliptic Keplerian orbits, with all trajectories having the same eccentricity $0<e<1$. See figure \ref{fig:lag3body}. When $e=0$, it's the triangular relative equilibrium.
\begin{figure}[ht]

\centering
  \includegraphics[width=0.5\textwidth]{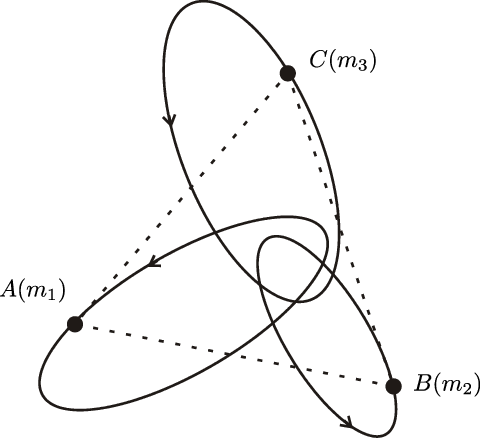}
  \caption{The elliptic Lagrange homographic solutions.}

 \label{fig:lag3body}
\end{figure}

For the strong force, we only have the triangular relative equilibrium, while the elliptic Lagrange homographic solutions do not exist. Because the only periodic solution of the Kepler problem for $\alpha>2$ is the circular orbit, and there are no elliptic Kepler orbits for $\alpha>2$. To the author's knowledge, we are not aware of any work concerning periodic or quasi-periodic solutions of the N-body problem for $\alpha>2$, except the relative equilibria and choreographies of the N-body problem. Our example of infinitely many transitions between $\mathcal K^\pm$ is motivated by the Sitnikov problem \cite{sit61}. The Sitnikov problem is a special case of the restricted 3-body problem that allows oscillatory motions. In particular, what we will consider here is also known as the MacMillan problem \cite{Macm1911}.  

\subsection{Setting of the MacMillan problem}
\begin{figure}[ht]

\centering
  \includegraphics[width=0.5\textwidth]{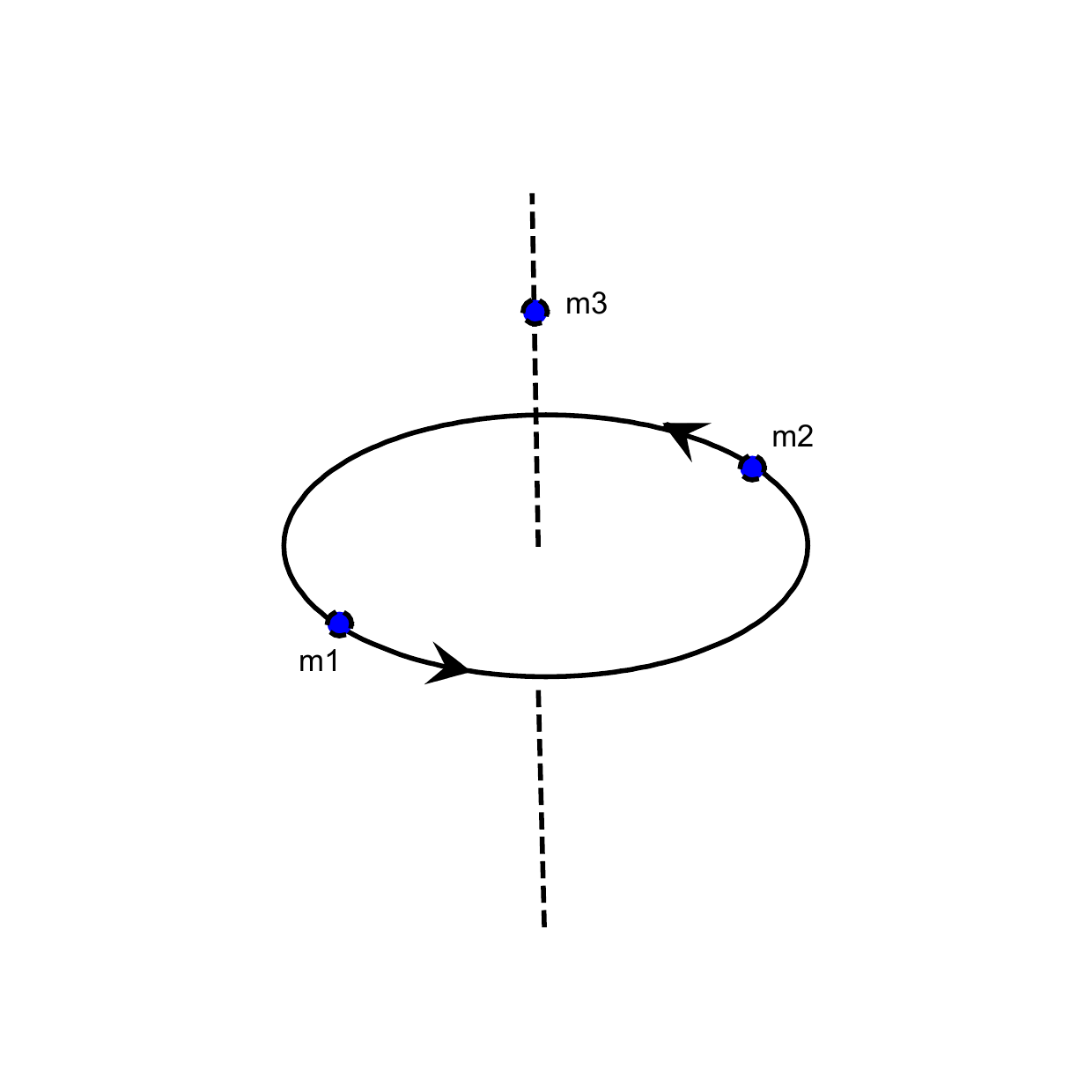}
  \caption{The MacMillan problem.}

 \label{fig:sitnikov}
\end{figure}

Let $\bf x_i=(x_i, y_i, z_i)$ be the position of three point masses $m_i$ in $\R^3$. The motion of the general 3-body problem is given by the differential equation
\begin{equation}
m_i\ddot{\bf x_i}=-\nabla_i U(\bf x)=-\alpha\sum_{j\neq i}\frac{m_im_j(\bf x_i-\bf x_j)}{|\bf x_i-\bf x_j|^{\alpha+2}},\quad i=1,2,3.
\end{equation} where $\bf x=(\bf x_1, \bf x_2, \bf x_3)$.

Let $m_1=m_2=m$, referred as the primary bodies, assume they move in a circular orbit around their center of mass. A massless body ($m_3=0$) moves (oscillates) along a straight line that is perpendicular to the orbital plane formed by the two equally massed primary bodies (cf. Figure \ref{fig:sitnikov}). Since $m_3=0$, its influence to the primary bodies are negligible. We may assume the primary bodies move in the $xy$-plane, and $m_3$ moves along the $z$-axis. Let's take $m=1/2$ and the radius of the circle is $r=1$, then the frequency of the circular motion is $\omega=\sqrt{\frac{\alpha}{2^{\alpha+2}}}$. Let $\bf x_3=(0,0,z_3)$,  the equation of motion for $m_3$ is given by 

\begin{equation}\label{eq:mac}\ddot{z_3}+\frac{\alpha z_3}{(\sqrt{1+z_3^2})^{\alpha+2}}=0,\end{equation}
which is a Hamiltonian system. Let $v=\dot{z}_3$, then the hamiltonian for $(z_3,v)$ is \begin{equation}H(z_3,v)=\frac{v^2}{2}-\frac{1}{(\sqrt{1+z_3^2})^{\alpha}}.\end{equation}
The level curves of $H(z_3,v)$ are illustrated in Figure \ref{fig:contourH}. $H(0,0)=-1$ is the global minimum and when $-1<H<0$, the level curves are closed which yield periodic solutions. Moreover, when $H(z_3,0)\to 0^-$, we have $|z_3|\to \infty$. That is, we find periodic solutions of the restricted 3-body problem with mutual distances $r_{12}=2$, and $r_{13}=r_{23}$ oscillates from $1$ to arbitrarily large. But $m_3=0$, and the primary bodies form a relative equilibrium, thus the threshold function $$K_\omega(\bf x(t))=\omega^2(x_1^2+y_1^2)-\frac{\alpha}{2^{\alpha+2}(x_1^2+y_1^2)^{\frac{\alpha}{2}}}=\omega^2-\frac{\alpha}{2^{\alpha+2}}=0,$$ 
for all time. We need to extend this system to positive mass for $m_3$.
\begin{figure}[ht]

\centering
  \includegraphics[width=0.5\textwidth]{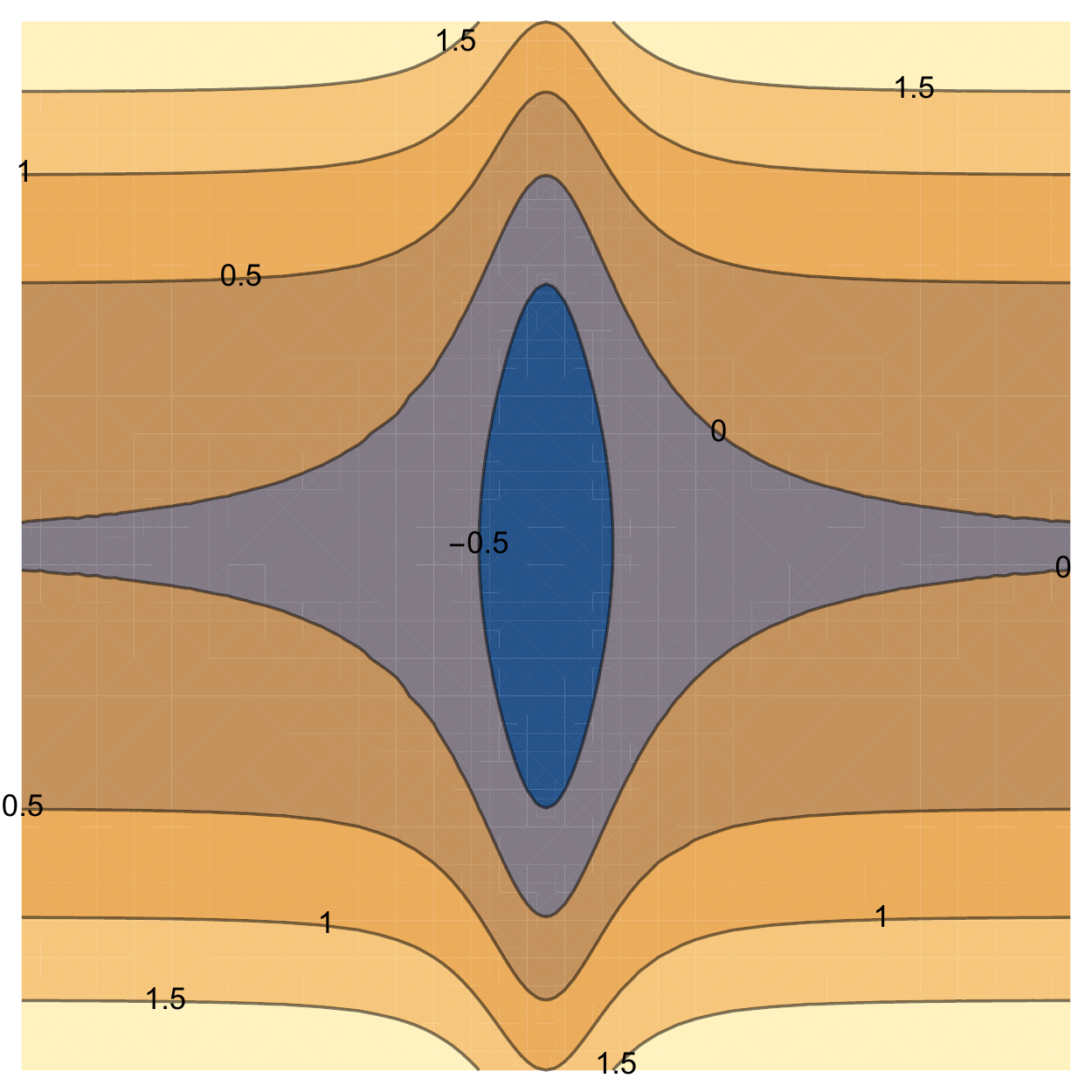}
  \caption{The contour plot for $H(z_3,v)$ with $\alpha=3$.}

 \label{fig:contourH}
\end{figure}

Now let the mass $m_3=\epsilon$. Because of the symmetry of the masses, there are motions satisfying the constraints:
\begin{equation}\nonumber(x_1, y_1, z_1)=(-x_2, -y_2, z_2),\end{equation}
\begin{equation} \nonumber x_3=y_3=\dot{x}_3=\dot{y}_3=0.\end{equation}
The center of mass is fixed at the origin, i.e. we always assume $z_1=-\epsilon z_3$ and $\dot{z}_1=-\epsilon\dot{z}_3$. The assumptions we make allow us to investigate the reduced set of differential equations:

\begin{equation}
\label{eq:epMac1}
\begin{split}
\ddot{x}_1&=-\alpha(\frac{x_1}{r_{12}^{\alpha+2}}+\frac{\epsilon x_1}{r_{13}^{\alpha+2}}),\\
\ddot{y}_1&=-\alpha(\frac{y_1}{r_{12}^{\alpha+2}}+\frac{\epsilon y_1}{r_{13}^{\alpha+2}}),\\
\ddot{z}_3&=-\frac{\alpha(1+\epsilon)z_3}{r_{13}^{\alpha+2}}.\\
\end{split}
\end{equation}
where $r_{12}=|\bf x_1-\bf x_2|=2\sqrt{x_1^2+y_1^2}$, and $r_{13}=|\bf x_1-\bf x_3|=\sqrt{x_1^2+y_1^2+((1+\epsilon)z_3)^2}$. When $\epsilon=0$, we have $z_1=0$. The primary bodies form a two-body problem and if they are in the circular motion with $x_1^2+y_1^2=1$, equation (\ref{eq:epMac1}) reduces to the MacMillan equation (\ref{eq:mac}). We will call (\ref{eq:epMac1}) the $\epsilon$-MacMillan problem.

The conserved energy of the $\epsilon$-MacMillan problem is

\begin{equation}
\label{eq:Eqqep}
\begin{split}&E(\bf x, \dot{\bf x};\epsilon)\\
&=\frac{1}{2}(\dot{x}_1^2+\dot{y}_1^2+\dot{z}_1^2)+\frac{\epsilon}{2}\dot{z}_3^2-(\frac{1}{2^{\alpha+2}(x_1^2+y_1^2)^{\frac{\alpha}{2}}}+\frac{\epsilon}{(x_1^2+y_1^2+(z_1-z_3)^2)^{\frac{\alpha}{2}}}),\\
&=\frac{1}{2}(\dot{x}_1^2+\dot{y}_1^2)-\frac{1}{2^{\alpha+2}(x_1^2+y_1^2)^{\frac{\alpha}{2}}}+\epsilon[\frac{1+\epsilon}{2}\dot{z}_3^2-\frac{1}{(x_1^2+y_1^2+((1+\epsilon)z_3)^2)^{\frac{\alpha}{2}}}].
\end{split}\end{equation}

The angular momentum is

\begin{equation}A(\bf x, \dot{\bf x}; \epsilon)=\sum_{i=1}^3 m_i \bf x_i\times \dot{\bf x}_i=(0, 0, x_1\dot{y}_1-y_1\dot{x}_1).\end{equation}
That is, the angular momentum is contributed by the primary bodies only. To make the computations concrete, we choose the frequency parameter for the $\epsilon$-MacMillan problem as $\omega=\sqrt{\frac{\alpha}{2^{\alpha+2}}}$, \footnote{If we choose a different frequency $\omega$, the computations seem to be more complicated.}  and we will restrict the solutions on the angular momentum level set with $|A(\bf x, \dot{\bf x}; \epsilon)|=|x_1\dot{y}_1-y_1\dot{x}_1|=\omega$. This is the angular momentum level for the $0$-MacMillan problem when the radius of the primary bodies is 1 and frequency is $\omega$. The energy for the relative equilibrium of the $\epsilon$-MacMillan problem with frequency $\omega$ is 

\begin{equation}
\label{eq:EstarO}
\begin{split}
E^*(\omega;\epsilon)&=(\frac{\alpha}{2}-1)(\epsilon+\frac{1}{2^{\alpha+2}})^{\frac{2}{\alpha+2}}(\frac{\omega^2}{\alpha})^{\frac{\alpha}{\alpha+2}},\\
&=(\frac{\alpha}{2}-1)(\epsilon+\frac{1}{2^{\alpha+2}})^{\frac{2}{\alpha+2}}\frac{1}{2^{\alpha}}.
\end{split}
\end{equation}
and $E^*(\omega;0)$ is the excited energy for the $0$-MacMillan problem.

The threshold function $K_\omega(\bf x; \epsilon)$ for the $\epsilon$-MacMillan problem is
\begin{equation}
\label{eq:Kqep}
\begin{split}
K_\omega(\bf x; \epsilon)&=\omega^2(x_1^2+y_1^2)-\frac{\alpha}{2^{\alpha+2}(x_1^2+y_1^2)^{\frac{\alpha}{2}}}+\epsilon[(1+\epsilon)\omega^2z_3^2-\frac{\alpha}{(x_1^2+y_1^2+(1+\epsilon)^2z_3^2)^{\frac{\alpha}{2}}}],\\
&=K_\omega(\bf x; 0)+\epsilon[(1+\epsilon)\omega^2z_3^2-\frac{\alpha}{(x_1^2+y_1^2+(1+\epsilon)^2z_3^2)^{\frac{\alpha}{2}}}].
\end{split}
\end{equation}

\subsection{Two reference equations for the $\epsilon$-MacMillan problem}

To study the motion of the $\epsilon$-MacMillan problem, we introduce two extreme cases. Namely the case when the third body $m_3$ is at rest at the origin, and the case when $m_3$ is infinitely far away from the origin. 

Suppose $z_3=\dot{z}_3=0$ then equation (\ref{eq:epMac1}) is equivalent to 
\begin{equation}
\label{eq:z3is0}
\ddot{\bf x}^0=\nabla U(\bf x^0;\epsilon),\quad U(\bf x^0;\epsilon)=\frac{1+2^{\alpha+1}\epsilon}{|\bf x^0|^\alpha}, \quad\bf x^0=(2x_1, 2y_1). \end{equation}
Suppose $z_3=\infty$, then equation (\ref{eq:epMac1}) is equivalent to 
\begin{equation}
\label{eq:z3is8}
\ddot{\bf x}^\infty=\nabla U(\bf x^\infty),\quad U(\bf x^\infty)=\frac{1}{|\bf x^\infty|^\alpha}, \quad\bf x^\infty=(2x_1, 2y_1).
\end{equation}
Note that we have used $\bf x^0$ and $\bf x^\infty$ to denote solutions for (\ref{eq:z3is0})(\ref{eq:z3is8}) specifically, and they are the horizontal relative position of the primary bodies. 



Now we present some comparisons between $\bf x^0$ and $\bf x^\infty$ with the restriction
\begin{equation}
\bf x^0\times\dot{\bf x}^0=\bf x^\infty\times\dot{\bf x}^\infty=c.
\end{equation} 

In polar coordinates $(r,\theta)$, the effective potentials of $\bf x^0$ and $\bf x^\infty$ are
\begin{equation}
\begin{split}
V_c^0(r)&=\frac{c^2}{2r^2}-\frac{1+2^{\alpha+1}\epsilon}{r^\alpha},\\
V_c^\infty(r)&=\frac{c^2}{2r^2}-\frac{1}{r^\alpha}.
\end{split}
\end{equation} 

The graph of $V_c^0(r)$ is below that of $V_c^\infty(r)$, see Figure \ref{fig:epKep}.
\begin{figure}
\centering
  \includegraphics[width=0.6\textwidth]{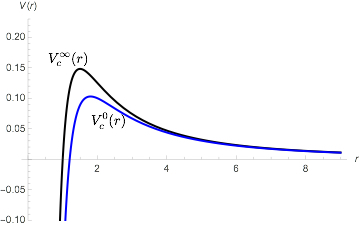}
 \caption{The effective potential $V_c^0(r)$ (blue) and $V_c^\infty(r)$ (black).}
 \label{fig:epKep}
\end{figure}

\begin{figure}
\centering
  \includegraphics[width=0.5\textwidth]{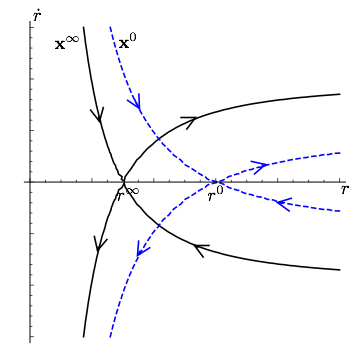}
 \caption{The threshold curves for phase portraits of $\bf x^0$ and $\bf x^\infty$ with $\bf x^0\times\dot{\bf x}^0=\bf x^\infty\times\dot{\bf x}^\infty=c$ in the $(r, \dot{r})$ phase plane. The black solid curve is for $\bf x^\infty$ and the threshold energy is $v_c^\infty$. The dashed blue curve is for $\bf x^0$ with threshold energy $v_c^0$. }
 \label{fig:epKepPP}
\end{figure}

The critical points of $V_c^0(r)$ and $V_c^\infty(r)$, i.e. the radius for the corresponding relative equilibrium, are
\begin{equation}
\begin{split}
r^0&=(\frac{\alpha(1+2^{\alpha+1}\epsilon)}{c^2})^{\frac{1}{\alpha-2}},\\
r^\infty&=(\frac{\alpha}{c^2})^{\frac{1}{\alpha-2}}.
\end{split}
\end{equation} 

The maximal values of $V_c^0(r)$ and $V_c^\infty(r)$, i.e. the energy for the corresponding relative equilibrium, are
\begin{equation}
\begin{split}
v_c^0&:=V_c^0(r^0)=\alpha^{\frac{2}{2-\alpha}}(\frac{1}{2}-\frac{1}{\alpha})c^{\frac{2\alpha}{\alpha-2}}(\frac{1}{1+2^{\alpha+1}\epsilon})^{\frac{2}{\alpha-2}},\\
v_c^\infty&:=V_c^\infty(r^\infty)=\alpha^{\frac{2}{2-\alpha}}(\frac{1}{2}-\frac{1}{\alpha})c^{\frac{2\alpha}{\alpha-2}}.
\end{split}
\end{equation} 

The phase portraits of $\bf x^0$ and $\bf x^\infty$ are illustrated in Figure \ref{fig:epKepPP}.

To facilitate our analysis for the $\epsilon$-MacMillan problem, we will take $c=4\omega$. Note we choose this value because $\bf x^0\times\dot{\bf x}^0=\bf x^\infty\times\dot{\bf x}^\infty=4(x_1\dot{y}_1-\dot{x}_1y_1)$. Moreover, we will have \begin{equation}\label{eq:4omega}r^{0}=2(1+2^{\alpha+1}\epsilon)^{\frac{1}{\alpha-2}},\,\,v_{4\omega}^0=4(\frac{\alpha}{2}-1)\frac{1}{2^{\alpha+2}}(\frac{1}{1+2^{\alpha+1}\epsilon})^{\frac{2}{\alpha-2}}.\end{equation} and \begin{equation}r^{\infty}=2,\,\,v_{4\omega}^\infty=4E^*(\omega;0)=4(\frac{\alpha}{2}-1)\frac{1}{2^{\alpha+2}}.\end{equation}
Note that $v_{4\omega}^0$ is strictly less than $4E^*(\omega;\epsilon)$ in (\ref{eq:EstarO}).


\subsection{Infinitely many transitions}

For the $\epsilon$-MacMillan problem (\ref{eq:epMac1}), and energy $E(\bf x, \dot{\bf x}; \epsilon)$ in (\ref{eq:Eqqep}), we restrict our solutions on the set \begin{equation}\mathcal S:=\{(\bf x, \dot{\bf x}) | E(\bf x, \dot{\bf x}; \epsilon)<\frac{1}{4}v_{4\omega}^0, |A|=\omega\}.\end{equation}
This is an invariant set of the $\epsilon$-MacMillan problem and $\frac{1}{4}v_{4\omega}^0$ is strictly less than $E^*(\omega;\epsilon)$, i.e. $\mathcal S$ is a subset of $\mathcal K:=\{(\bf x, \dot{\bf x}) | E(\bf x, \dot{\bf x}; \epsilon)<E^*(\omega;\epsilon)\}$. Let 
\begin{equation}
\begin{split}
\mathcal S_+&:=\{(\bf x, \dot{\bf x})\in\mathcal S | 2\sqrt{x_1^2+y_1^2}>r^0\},\\
\mathcal S_-&:=\{(\bf x, \dot{\bf x})\in\mathcal S | 2\sqrt{x_1^2+y_1^2}<r^0\}.\\
\end{split}
\end{equation}
where $r^0$ is defined in (\ref{eq:4omega}).

\begin{lemma}
The sets $\mathcal S_{\pm}$ are invariant for the $\epsilon$-MacMillan problem.
\end{lemma}
\begin{proof}
Let \[E(x_1, y_1, \dot{x}_1, \dot{y}_1;\epsilon):=\frac{1}{2}(\dot{x}_1^2+\dot{y}_1^2)-\frac{1+\epsilon}{2^{\alpha+2}(x_1^2+y_1^2)^{\frac{\alpha}{2}}},\] which is $E(\bf x, \dot{\bf x}; \epsilon)$ by setting $z_3=\dot{z_3}=0$. Thus \[E(x_1, y_1, \dot{x}_1, \dot{y}_1;\epsilon)\leq E(\bf x, \dot{\bf x}; \epsilon)<\frac{1}{4}v_{4\omega}^0.\]
By Figure \ref{fig:MacM2}, note that $r=2\sqrt{x_1^2+y_1^2}$, we get the invariance of the sets $\mathcal S_{\pm}$. Moreover, we have $\mathcal S_+$ is the region $E$, $\mathcal S_-=A\cup B\cup C\cup D$. $A\cup D$ is forward time invariant. $B$ is backward time invariant, see Figure \ref{fig:MacM2}.
\end{proof}

\begin{figure}
\centering
  \includegraphics[width=0.46\textwidth]{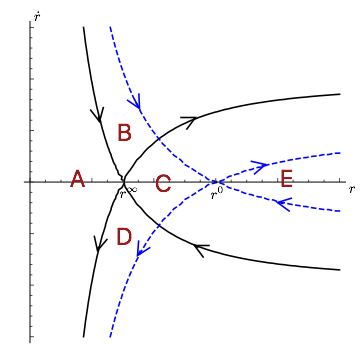}
 \caption{Black curve corresponds to the level curve $E(x_1, y_1, \dot{x}_1, \dot{y}_1;0)=E^*(\omega;0)=\frac{1}{4}v_{4\omega}^\infty$ and blue curve corresponds to the level curve $E(x_1, y_1, \dot{x}_1, \dot{y}_1; \epsilon)=\frac{1}{4}v_{4\omega}^0$ in the $(r, \dot{r})$ space with $x_1\dot{y}_1-\dot{x}_1y_1=\omega$. Moreover, $\mathcal S_+=E$ and $\mathcal S_-=A\cup B\cup C\cup D$ are invariant. $A\cup D$ is forward time invariant. $B$ is backward time invariant.}
 \label{fig:MacM2}
\end{figure}

Let's focus on the region $C$ and seek for a solution that stays in $C$. Roughly speaking, when $z_3$ is far away, the motion of $(r, \dot{r})$ is predicted by the black threshold curve; when $z_3$ is close to zero, the motion of $(r, \dot{r})$ is predicted by the blue threshold curve, see Figure \ref{fig:MacM3}. Suppose $z_3(0)=O(1/\epsilon), \dot{z}_3(0)=0$, and $r^\infty<r(0)<r^0, \dot{r}(0)=0$, then $(r,\dot{r})$ tends to go along the black curve. As $z_3$ approaches zero, $(r,\dot{r})$ tends to go along the blue curve. Then when $z_3$ passes zero and continues to $O(-1/\epsilon)$, $(r,\dot{r})$ tends to go along the black curve, etc. This is a solution with infinitely many transitions for $K_\omega(\bf x(t); \epsilon)$ from positive to negative. More specifically,

\begin{equation}
\begin{split}
K_\omega(\bf x(t); \epsilon)=&\omega^2(x_1^2+y_1^2)-\frac{\alpha}{2^{\alpha+2}(x_1^2+y_1^2)^{\frac{\alpha}{2}}},\\&+\epsilon[(1+\epsilon)\omega^2z_3^2-\frac{\alpha}{(x_1^2+y_1^2+(1+\epsilon)^2z_3^2)^{\frac{\alpha}{2}}}].
\end{split}
\end{equation}
When $r^\infty<2\sqrt{x_1^2+y_1^2}<r^0$, we have \[\omega^2(x_1^2+y_1^2)-\frac{\alpha}{2^{\alpha+2}(x_1^2+y_1^2)^{\frac{\alpha}{2}}}>0,\] and \[\omega^2(x_1^2+y_1^2)-\frac{\alpha(1+2^{\alpha+2}\epsilon)}{2^{\alpha+2}(x_1^2+y_1^2)^{\frac{\alpha}{2}}}<0.\] Thus easy to see $K_\omega(\bf x(t); \epsilon)$ is positive when $z_3=O(\pm1/\epsilon)$ and negative when $z_3=0$.

\begin{figure}
\centering
  \includegraphics[width=0.46\textwidth]{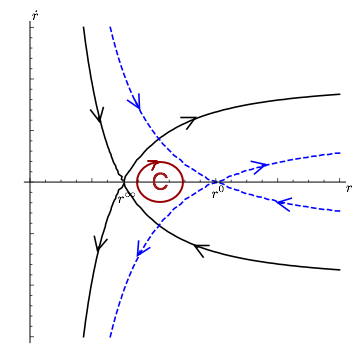}
 \caption{A solution where $z_3$ oscillates, $(r, \dot{r})$ stays in $C$.}
 \label{fig:MacM3}
\end{figure}

\section{Some comments and future plans}

\subsection{Excited energy and the frequency.}
When the frequency is small, the excited energy $E^*(\omega)$ is small, and it goes to zero if $\omega\to 0^+$, see Proposition \ref{prop:0infty} below. Since the angular momentum for a relative equilibrium is $|A|=\omega I(\bf x)$, if we fix the angular momentum, the frequency $\omega$ can exhaust all the positive values, thus the minimum energy for all relative equilibria with fixed angular momentum will be \[\mathcal E_{|A|}^*=\lim_{\omega\to 0+} E^*(\omega)=0.\]
Therefore, we see again why we do not use the angular momentum as the parameter when defining the excited energy. 

Moreover, any solution with non-zero angular momentum can be characterized in the way as summarized in Theorem \ref{thm:ref_nbd}. The reason is that $E^*(\omega)$ is increasing and goes to infinity as $\omega \to\infty$, see Lemma \ref{lem:inc} and Proposition \ref{prop:0infty}, thus any trajectory will have energy less than $E^*(\omega)$ for some $\omega$.

\begin{lemma}
\label{lem:inc}
If $\omega_1\leq \omega_2$, then $E^*(\omega_1)\leq E^*(\omega_2)$.
\end{lemma}
\begin{proof}
From Lemma \ref{lem:intermediate}, we have 
\[\inf \{-U(\bf x): K_\omega(\bf x)=0\}=\inf \{-U(\bf x): K_\omega(\bf x)\leq0\}.\]
When $\omega_1\leq \omega_2$, we have $\{K_{\omega_2}(\bf x)= 0\}\subseteq\{K_{\omega_1}(\bf x)\leq0\}$. Thus,
\begin{equation}
\begin{split}
E^*(\omega_1)&=(\frac{\alpha}{2}-1)\inf \{-U(\bf x): K_{\omega_1}(\bf x)=0\},\\
     &=(\frac{\alpha}{2}-1)\inf \{-U(\bf x): K_{\omega_1}(\bf x)\leq0\},\\
     &\leq(\frac{\alpha}{2}-1)\inf \{-U(\bf x): K_{\omega_2}(\bf x)=0\},\\
     &=E^*(\omega_2).
\end{split}
\end{equation}
\end{proof}

\begin{proposition}
\label{prop:0infty}
$$\lim_{\omega\to0^+} E^*(\omega)=0, \quad \lim_{\omega\to\infty} E^*(\omega)=+\infty.$$
\end{proposition}
\begin{proof}
When $\omega=0$, $K_0(\bf x)=U(\bf x)=0$, thus $E_0(\bf x)=-(\alpha/2-1)U(\bf x)=0$, and we have $\lim_{\omega\to0^+} E^*(\omega)=0$. 

We compute the limit for $\omega\to\infty$. From the previous lemma, we know $E^*(\omega)$ is non-decreasing, so the limit exists. Suppose $\lim_{\omega\to\infty} E^*(\omega)=C$ where $0<C<\infty$. i.e.
\begin{equation}
\begin{split}
\lim_{\omega\to\infty}E^*(\omega)&=\lim_{\omega\to\infty}(\frac{\alpha}{2}-1)\inf \{-U(\bf x): K_{\omega}(\bf x)=0\}=C.
\end{split}
\end{equation}
This is not possible under the constraint $K_\omega=\omega^2 I+\alpha U=0$ because of the following claim. 

\textbf{Claim:} If $-U(\bf x)\leq c$, then $I(\bf x)\geq \frac{m^2}{M}(m^2/c)^{2/\alpha}$.

\textbf{Proof of claim:}Let $M=m_1+\cdots+m_N$, and $m=\min\{m_1, \cdots, m_N\}$. If $U\leq c$, then \[\min_{i<j} r_{ij}^\alpha\geq \frac{m^2}{c},\]
thus \[I(\bf x)\geq \frac{m^2}{M}(m^2/c)^{2/\alpha}.\]
\end{proof}

\subsection{Excited state for the equal mass 3-body problem}

Central configurations and relative equilibria of the 3-body problem are well-known. Namely, the Euler (co-linear) and Lagrange (equilateral triangle) relative equilibria \cite{MeOf17}. In this section we compute the excited energy for 3-body problem with equal masses.

\begin{proposition}
\label{prop:eq3bdex}
Let $\alpha>2$, for the 3-body problem with equal masses, the excited state is the co-linear relative equilibrium.
\end{proposition}
\begin{proof}
 Without loss of generality, let the masses be $m_1=m_2=m_3=1$. 

\textbf{Co-linear R.E.} Let $x_1<x_2<x_3$ and $x=x_2-x_1$ and $y=x_3-x_2$. The equation for co-linear R.E. is
\begin{equation}\omega^2x=\alpha(\frac{m_1+m_2}{x^{\alpha+1}}-\frac{m_3}{y^{\alpha+1}}+\frac{m_3}{(x+y)^{\alpha+1}}),\end{equation}
\begin{equation}\omega^2 y=\alpha(\frac{m_1}{(x+y)^{\alpha+1}}-\frac{m_1}{x^{\alpha+1}}+\frac{m_2+m_3}{y^{\alpha+1}}).\end{equation}
when the masses are equal, we have $x=y$ and \[x=[\frac{\alpha}{\omega^2}(1+\frac{1}{2^{\alpha+1}})]^{\frac{1}{\alpha+2}}.\]
The energy of the co-linear R.E. is
\[E_{linear}=-(\frac{\alpha}{2}-1)U=(\frac{\alpha}{2}-1)(2+\frac{1}{2^\alpha})x^{-\alpha}=2(\frac{\alpha}{2}-1)(1+\frac{1}{2^{\alpha+1}})^{\frac{2}{\alpha+2}}(\frac{\alpha}{\omega^2})^{-\frac{\alpha}{\alpha+2}}.\]

\textbf{Triangle R.E.} The mutual distances are \[r_{12}=r_{13}=r_{23}=(\frac{\alpha M}{2\omega^2})^{\frac{1}{\alpha+2}}.\]
The energy of triangle R.E. is
\[E_{triangle}=-(\frac{\alpha}{2}-1)U=(\frac{\alpha}{2}-1)3r^{-\alpha}=2(\frac{\alpha}{2}-1)(1+\frac{1}{2})^{\frac{2}{\alpha+2}}(\frac{\alpha}{\omega^2})^{-\frac{\alpha}{\alpha+2}}.\]
Therefore, $E_{linear}<E_{triangle}$.
\end{proof}

By Moulton's Theorem, for $N\geq 3$, there are always $N!/2$ co-linear relative equilibria for some fixed $\omega$.

\begin{theorem}[Moulton \cite{moulton}]
In the co-linear N-body problem, for any choice of $N$ positive masses there are exactly $N!/2$ central configurations. One for each ordering of the particles modulus a rotation by $\pi$.
\end{theorem}

\begin{conjecture}
The (first) excited states are the co-linear Relative equilibria for general masses.
\end{conjecture}

\subsection{Invariance of $\mathcal K^{\pm}(\omega)$ and the angular momentum}

We are aware of the fact that $\mathcal K(\omega)=\{(\bf x, \dot{\bf x}): E(\bf x, \dot{\bf x})<E^*(\omega)\}$ still contains relative equilibria and this could be the reason why $\mathcal K^\pm(\omega)$ are not invariant. In the PDE examples, the energy constraint is sufficient to exclude all the solitons in the set. To tackle this problem, we could add a lower bound on the level of the angular momentum like we did for the two-body problem, i.e. consider the set \begin{equation}\mathcal K(\omega)=\{(\bf x, \dot{\bf x}): E(\bf x, \dot{\bf x})<E^*(\omega), |A(\bf x, \dot{\bf x})|\geq A^*(\omega)\}.\end{equation}
The strongest choice of the lower bound $A^*(\omega)$ would be
\begin{equation}A^*(\omega):=\sup\{\omega I(\bf x) | K_\omega(\bf x)=0\}.\end{equation}
but $A^*(\omega)=\infty$ as can be seen in the proof of Lemma \ref{lem:ground}. The next choice would be $A^*(\omega):=\omega I(\bf q)$ where $\bf q$ is the configuration so that $E_\omega(\bf q)=E^*(\omega)$. To show that this condition excludes relative equilibria is highly related with the problem of the central configurations of the N-body problem. For the equal mass three-body problem, we are able to show that this choice works, see Proposition \ref{prop:eq3bd}.

From Proposition \ref{prop:eq3bdex}, we know when $m_1=m_2=m_3=1$, the colinear $\omega$-R.E. has smaller energy than the triangular $\omega$-R.E. Now we want to compare their angular momentum. The energy and angular momentum of the co-linear R.E. are
\begin{equation}
\label{eq:Estar}
E^*(\omega)=E_{linear}(\omega)=(\frac{\alpha}{2}-1)(2+\frac{1}{2^\alpha})x^{-\alpha}=2(\frac{\alpha}{2}-1)(1+\frac{1}{2^{\alpha+1}})^{\frac{2}{\alpha+2}}(\frac{\alpha}{\omega^2})^{-\frac{\alpha}{\alpha+2}},\end{equation}
 \begin{equation}
 \label{eq:Astar}
 A^*(\omega)=A_{linear}(\omega)=\omega I(\bf q)=\omega 2x^2=2[\alpha(1+\frac{1}{2^{\alpha+1}})]^{\frac{2}{\alpha+2}}\omega^{\frac{\alpha-2}{\alpha+2}}.
\end{equation}

\begin{proposition}
\label{prop:eq3bd}
For the three-body problem with $m_1=m_2=m_3=1$, let $E^*(\omega)$ and  $A^*(\omega)$ be as in (\ref{eq:Estar})(\ref{eq:Astar}), let
\[\mathcal K(\omega)=\{E(\bf x,\dot{\bf x})<E^*(\omega), |A|\geq A^*(\omega)\},\]
then $\mathcal K(\omega)$ does not contain any relative equilibria.
\end{proposition}

\begin{proof} It is easy to see that all the co-linear R.E. are excluded, let's see if the triangular R.E. is also excluded. The energy and angular momentum of a triangular R.E. is
\begin{equation}E_{triangle}(\omega)=(\frac{\alpha}{2}-1)3r^{-\alpha}=2(\frac{\alpha}{2}-1)(1+\frac{1}{2})^{\frac{2}{\alpha+2}}(\frac{\alpha}{\omega^2})^{-\frac{\alpha}{\alpha+2}},\end{equation}
\begin{equation}A_{triangle}(\omega)=\omega r^2=\omega (\frac{3\alpha}{2\omega^2})^{\frac{2}{\alpha+2}}=[\alpha(1+\frac{1}{2})]^{\frac{2}{\alpha+2}}\omega^{\frac{\alpha-2}{\alpha+2}}.
\end{equation}
Fix $\omega$, let's see whether we can find $\omega_1$, so that the triangular $\omega_1$-R.E. is in the set $\mathcal K(\omega)$.
From $E_{triangle}(\omega_1)<E_{linear}(\omega)$, we get
\begin{equation}
\omega_1<(\frac{1+\frac{1}{2^{\alpha+1}}}{1+\frac{1}{2}})^{\frac{1}{\alpha}}\omega<\omega.
\end{equation}
From $A_{triangle}(\omega_1)\geq A_{linear}(\omega)$, we get
\begin{equation}
\omega_1\geq 2^{\frac{\alpha+2}{\alpha-2}}(\frac{1+\frac{1}{2^{\alpha+1}}}{1+\frac{1}{2}})^{\frac{2}{\alpha-2}}\omega>(\frac{2+\frac{1}{2^{\alpha}}}{3/2})^{\frac{2}{\alpha-2}}\omega>\omega.
\end{equation}
So there is no triangular R.E. in the set $\mathcal K(\omega)$ either.
\end{proof}

It seems not an easy task to show this for general masses when $N=3$, let alone when $N\geq 4$. However, this provides a good direction for us, and we will work on these problems in our subsequent work.

\section{acknowledgement}

We are grateful to Kenji Nakanishi, Ernesto Perez-Chavela and Cristina Stoica for insightful discussions. We would like to thank Belaid Moa for the numerical simulations on the MacMillan problem. The first author is partially supported by the NSERC grant. The second author is supported by the NSERC grant No. 371637-2014.


\begin{thebibliography}{FF}
\bibitem{AkIbKi18} T. Akahori, S. Ibrahim, H. Kikuchi: Linear instability and nondegeneracy of ground state for combined power-type nonlinear scalar field equations with the Sobolev critical exponent and large frequency parameter. \emph{https://arxiv.org/pdf/1810.12363} (2018)
\bibitem{AkIbIk18}T. Akahori, S. Ibrahim, N. Ikoma, H. Kikuchi, H. Nawa: Uniqueness and nondegeneracy of ground states to nonlinear scalar field equations involving the Sobolev critical exponent in their nonlinearities for high frequencies. \emph{https://arxiv.org/abs/1801.08696} (2018)
\bibitem{AkIb18} T. Akahori, S. Ibrahim, H. Kikuchi, H. Nawa: Global dynamics above the ground state energy for the combined power-type nonlinear Schr\"odinger equations with energy-critical growth at low frequencies. \emph{To appear in Memoirs of the A.M.S.}
\bibitem{FlKn18}S. Fleischer, A. Knauf: Improbability of Collisions in n-Body Systems. arxiv.org/abs/1802.08564 (2018)
\bibitem{Gor75}W. B. Gordon: Conservation dynamical systems involving strong forces. \emph{Trans. A. M. S. 204,  113-135}(1975)
\bibitem{IbMaNa11}S. Ibrahim, N. Masmoudi, K. Nakanishi: Scattering threshold for the focusing nonlinear Klein-Gordon equation. \emph{Anal. PDE Vol. 4, No. 2, 405-460} (2011) 
\bibitem{LJ24}J. E. Lennard-Jones: On the determination of Molecular Fields. \emph{Proc. R. Soc. Lond. A, 106(738), 463-477}(1924)
\bibitem{Macm1911}W. MacMillan: An integrable case in the restricted problem of three bodies. Astron. J. 27, 11-13 (1911)
\bibitem{MeOf17}K. R. Meyer, D. Offin: Introduction to Hamiltonian Dynamcial Systsems and the N-body Problem, Third Ed., ~Springer-Verlag, (2017)
\bibitem{moulton} F. R. Moulton: The Straight Line Solutions of  N bodies. \emph{Ann. of Math. 12, 1-17} (1910) 
\bibitem{NaSc91} K. Nakanishi, W. Schlag: Invariant manifolds and dispersive Hamiltonian evolution equations. \emph{European Mathematics Society} (2011)
\bibitem{Na17} K. Nakanishi: Global dynamics below excited solitons for the nonlinear Schr\"odinger equation with a potential. \emph{J. Math. Soc. Japan Vol. 69, No. 4 1353-1401} (2017)
\bibitem{PaSa75}L. E. Payne, D. H. Sattinger: Saddle points and instability of nonlinear hyperbolic equations. \emph{Isreal J. Math., 22, 273-303} (1975)
\bibitem{Saa71}D. Saari: Improbability of collisions in Newtonian gravitational systems. \emph{Trans. AMS  162, 267-271} (1971)
\bibitem{Saa73}D. Saari: Improbability of collisions in Newtonian gravitational systems. II \emph{Trans. AMS  181, 351-368} (1973)
\bibitem{Sch12}D. J. Scheeres: Minimum energy configuration in the N-body problem and the Celestial Mechanics of Granular Systems. \emph{Celes. Mech. Dyn. Astr. 113, 291-320} (2012)
\bibitem{SiMo71}C. L. Siegel, J. Moser: Lectures on Celestial Mechanics, Springer-Verlag, New York, Heidelberg, Berlin, (1971)
\bibitem{sit61}K. Sitnikov: The Existence of Oscillatory Motions in the Three-Body Problem. \emph{Soviet Physics Doklady, 5:647} (1961)
\bibitem{smale00}S. Smale: Mathematical problems for the next century, in Mathematics: Frontiers and Per-
spectives, ed. V. Arnold, M. Atiyah, P. Lax, and B. Mazur, American Math. Soc. 271-294 (2000).
\bibitem{Wintner}A. Wintner: The analytical foundations of celestial mechanics. \emph{Princeton University Press} (1941)
\bibitem{xia92}Z. Xia: The Existence of Non-collision Singularities in Newtonian Systems. \emph{Ann. of Math Vol.135, No.3, 411-468} (1992) 
\bibitem{SaaXia96} D. Saari, Z. Xia: Singularities in the Newtonian N-body Problem. \emph{Hamiltonian dynamics and celestial mechanics (Seattle, WA, 1995), Contemp. Math., 198:21-30} (1996)
\bibitem{Zei1908} H. Von Zeipel: Sur les Singularit\'es du Probl\'eme des n Corps. \emph{Arkiv f$\ddot{u}$r Mat., Astr. och Fysik 32, 1-4} (1908)














\end{thebibliography}
\end{document}